\documentclass[a4paper,12pt]{article}

\usepackage[utf8]{inputenc}
\usepackage[english]{babel}
\usepackage{sectsty}
\usepackage{mathtools,amssymb, amsthm, amsfonts, stmaryrd, mathpazo}
\usepackage[backend=biber]{biblatex}
\usepackage{graphicx, subfig}
\usepackage{hhline}
\usepackage{comment}
\usepackage{rotating}
\usepackage{geometry}
\usepackage{titlesec}
\usepackage[absolute]{textpos}
\usepackage{tikz, tikz-qtree, tikz-cd}
\tikzset{
    lablv/.style={anchor=south, rotate=90, inner sep=.5mm},
    labld/.style={anchor=south, rotate=150, inner sep=.8mm},
    lablad/.style={anchor=north, rotate=40, inner sep=1.2mm}
}
\usepackage{pgffor,etoolbox}
\usepackage{pgfplotstable}
\usepackage{longtable}
\usepackage{booktabs}
\pgfplotsset{compat=1.18}
\usepackage{fp}
\usepackage{braket}
\usepackage{xstring}
\usepackage{youngtab ,ytableau}
\usepackage{enumitem}
\usepackage{fancyhdr}
\usepackage{colortbl}
\usepackage{color}
\definecolor{valentia}{rgb}{233,78,82}
\definecolor{titleblue}{rgb}{114, 146, 162}
\usepackage{array}
\usepackage{makecell}
\usepackage{float}
\usepackage{phoenician}
\usepackage{mdframed}
\usepackage{hyperref}
\usepackage{upgreek}
\usepackage{authblk}
\usepackage{multirow}
\usepackage{multicol}
\usepackage{scrextend}
\DeclareFontFamily{U}{mathb}{}
\DeclareFontShape{U}{mathb}{m}{n}{<-5.5> mathb5 <5.5-6.5> mathb6
<6.5-7.5> mathb7 <7.5-8.5> mathb8 <8.5-9.5> mathb9 <9.5-11> mathb10
<11-> mathb12}{}
\DeclareSymbolFont{mathb}{U}{mathb}{m}{n}
\DeclareFontSubstitution{U}{mathb}{m}{n}
\DeclareMathSymbol{\blackdiamond}{\mathbin}{mathb}{"0C}

\DeclareMathAlphabet{\mathscr}{U}{rsfso}{m}{n}
\DeclareMathSymbol{\smin}{\mathbin}{AMSa}{"39}
\usepackage{authblk}

\newcommand{\Addresses}{{
  \bigskip
  \footnotesize

  \textsc{Institut Montpelliérain Alexander Grothendieck, Université Montpellier 2,  Place Eugène Bataillon, 34095 MONTPELLIER Cedex, FRANCE.}\par\nopagebreak
  \textit{E-mail address}: \texttt{raphael.paegelow@umontpellier.fr}
}}

\title{The fixed point locus of the smooth Jordan quiver variety under the action of the finite subgroups of $\SL_2(\C)$}
\author{Raphaël Paegelow}

\begin{document}

\theoremstyle{definition}
\newtheorem{deff}{Definition}[section]
\newtheorem{nota}[deff]{Notation}
\newtheorem{rmq}[deff]{Remark}
\newtheorem{ex}[deff]{Example}

\theoremstyle{plain}

\newtheorem{thm}[deff]{Theorem}
\newtheorem{cor}[deff]{Corollary}
\newtheorem{prop}[deff]{Proposition}
\newtheorem{lemme}[deff]{Lemma}
\newtheorem{fait}[deff]{Fact}
\newtheorem{conj}[deff]{Conjecture}
\newtheorem*{theorem1}{Theorem 1}
\newtheorem*{theorem2}{Theorem 2}
\newtheorem*{theorem3}{Theorem 3}
\newtheorem*{cor1}{Corollary}


\newcommand{\TODO}{\colorbox{red}{\textbf{{TODO}}}}

\newcommand{\DMQ}{\overline{Q^0_{\Gamma}}}
\newcommand{\DMQGf}{\overline{Q_{G^f}}}
\newcommand{\DMQf}{\overline{Q_{\Gamma^f}}}
\newcommand{\DMQJf}{\overline{Q_{\bullet^f}}}

\newcommand{\QVG}{\mathcal{M}^{\Gamma}}
\newcommand{\tQVG}{\tilde{\mathcal{M}}^{\Gamma}}
\newcommand{\QV}{\mathcal{M}}
\newcommand{\tQV}{\tilde{\mathcal{M}}}
\newcommand{\QVTM}{\mathcal{M}^{\Gamma}_{\boldsymbol{\theta}}(M)}
\newcommand{\QVT}{\mathcal{M}_{\theta}}
\newcommand{\QVTdef}{\mathcal{M}^{\lb}_{\theta}}
\newcommand{\tQVTdef}{\tilde{\mathcal{M}}^{\Gamma}_{\boldsymbol{\theta}, \lb}}
\newcommand{\tQVT}{\tilde{\mathcal{M}}^{\Gamma}_{\boldsymbol{\theta}}}
\newcommand{\tQVTM}{\tilde{\mathcal{M}}^{\Gamma}_{\boldsymbol{\theta}}(M)}
\newcommand{\QVTMs}{\mathcal{M}^{\Gamma}_{\boldsymbol{\theta}}(M)}
\newcommand{\tQVTMs}{\tilde{\mathcal{M}}^{\Gamma}_{\boldsymbol{\theta}}(M)}
\newcommand{\tQVTMb}{\tilde{\mathcal{M}}^{\Gamma}_{\boldsymbol{\theta},\lb\delta}(M)}
\newcommand{\QVTMb}{\mathcal{M}^{\lb\delta}_{\theta}(M)}
\newcommand{\QVTMdef}{\mathcal{M}^{\lb\delta}(M)}
\newcommand{\tQVTMdef}{\tilde{\mathcal{M}}^{\lb\delta}(M)}
\newcommand{\QVTd}{\mathcal{M}_{\theta}(d)}
\newcommand{\tQVTd}{\tilde{\mathcal{M}}_{\theta}(d)}
\newcommand{\QVJn}{\Mbullet_{\hspace{0.3em}1}(n)}
\newcommand{\QVJnG}{\Mbullet_{\hspace{0.3em}1}(n)^{\Gamma}}
\newcommand{\QVJndefu}{\Mbullet^{\hspace{0.3em}1}(n)}
\newcommand{\QVJnGdef}{\Mbullet^{\hspace{0.3em}1}(n)^{\Gamma}}
\newcommand{\QVJns}{\Mbullet_{\hspace{0.3em}\theta}(n)}
\newcommand{\QVJndef}{\Mbullet^{\hspace{0.3em}\lb}(n)}
\newcommand{\QVJnb}{\Mbullet^{\hspace{0.3em}\lb}_{\hspace{0.3em}\theta}(n)}
\newcommand{\Mbullet}{\mathcal{M}\hspace{-0.75em}\raisebox{0.8ex}{\scalebox{0.8}{$\bullet$}}}

\newcommand{\Hom}{\mathrm{Hom}}
\newcommand{\Irr}{\mathrm{Irr}}
\newcommand{\GL}{\mathrm{GL}}
\newcommand{\Tr}{\mathrm{Tr}}
\newcommand{\Aut}{\mathrm{Aut}}
\newcommand{\mult}{\mathrm{mult}}
\newcommand{\module}{\mathrm{mod}}
\newcommand{\Dyn}{\mathrm{Dyn}}
\newcommand{\End}{\mathrm{End}}
\newcommand{\SL}{\mathrm{SL}}
\newcommand{\SU}{\mathrm{SU}}
\newcommand{\G}{\mathrm{G}}
\newcommand{\op}{\mathrm{op}}
\newcommand{\RepGa}{\mathbf{McK}}
\newcommand{\Fr}{\mathrm{Fr}}
\newcommand{\Reg}{\mathrm{Reg}}
\newcommand{\id}{\mathrm{id}}
\newcommand{\Res}{\mathrm{Res}}
\newcommand{\Ind}{\mathrm{Ind}}
\newcommand{\Repm}{\mathbf{Rep}}

\newcommand{\Hk}{\mathcal{H}}
\newcommand{\Ck}{\mathscr{X}}
\newcommand{\Xk}{\mathcal{X}}

\newcommand{\Rep}{\mathrm{Rep}}
\newcommand{\RepG}{\mathrm{Rep}}
\newcommand{\RepQ}{\mathcal{R}}
\newcommand{\RepGn}{\mathrm{Rep}_{n}}
\newcommand{\RepGk}{\mathrm{Rep}_{k}}
\newcommand{\CharG}{\mathrm{Char}}
\newcommand{\CharGn}{\mathrm{Char}_{n}}
\newcommand{\CharGk}{\mathrm{Char}_{k}}
\newcommand{\Xstd}{X_{\mathrm{std}}}

\newcommand{\BD}{BD}
\newcommand{\Gamman}{\Gamma_n}
\newcommand{\Dt}{\tilde{D}}
\newcommand{\LDt}{L(\Dt_{l+2})}
\newcommand{\Dtc}{\tilde{D}}
\newcommand{\Ct}{\tilde{C}}
\newcommand{\Dtl}{\tilde{D}_{l+2}}
\newcommand{\WaD}{W({\tilde{D}_{l+2}})}
\newcommand{\WD}{W({D_{l+2})}}
\newcommand{\IG}{\mathrm{Irr}}

\newcommand{\Qc}{Q^{\vee}}
\newcommand{\QDts}{Q(\Dt_{l+2}^{\sigma})[0^+]}
\newcommand{\deltD}{\delta^{\widetilde{BD}_{2l}}}
\newcommand{\av}{\alpha^{\vee}}
\newcommand{\ABDz}{\mathcal{A}^{n,\T_1}_{\BD_{\ell}}}

\newcommand{\Wa}{W_{\Gamma}^{\text{aff}}}

\newcommand{\Taut}{\mathcal{T}_{d}}
\newcommand{\TautT}{\tilde{\mathcal{T}}_{d}}

\newcommand{\bigzero}{\mbox{\normalfont\Large\bfseries 0}}
\newcommand{\rvline}{\hspace*{-\arraycolsep}\vline\hspace*{-\arraycolsep}}
\newcommand{\C}{\mathbb{C}}
\newcommand{\Z}{\mathbb{Z}}
\newcommand{\R}{\mathbb{R}}
\newcommand{\Q}{\mathbb{Q}}
\newcommand{\lb}{\lambda}
\newcommand{\iso}{\xrightarrow{\,\smash{\raisebox{-0.65ex}{\ensuremath{\scriptstyle\sim}}}\,}}
\newcommand{\Sk}{\mathfrak{S}}
\newcommand{\Uk}{\mathcal{U}}
\newcommand{\T}{\mathbb{T}}
\newcommand{\Pro}{\mathscr{P}}
\newcommand{\Resol}{\mathscr{R}}
\newcommand{\Pnlo}{\mathcal{P}_{n,\ell}^{o}}
\newcommand{\IC}{\mathcal{IC}}
\newcommand{\gl}{\mathrm{g}_{\ell}}
\newcommand{\gG}{\mathrm{g}_{\Gamma}}
\newcommand{\gk}{\mathrm{g}_k}
\newcommand{\gll}{\mathrm{g}_{2l}}
\newcommand{\Cln}{C_{\ell,n}}
\newcommand{\Ctwo}{{(\C^2)}}
\newcommand{\oando}{\theta,\lb}
\newcommand{\Ch}{\mathfrak{C}}
\newcommand{\Alf}{\mathfrak{A}}
\newcommand*{\transp}[2][-3mu]{\ensuremath{\mskip1mu\prescript{\smash{\mathrm t\mkern#1}}{}{\mathstrut#2}}}

\newcommand{\qphoe}{\textphnc{q}}
\newcommand{\lphoe}{\textphnc{l}}
\newcommand{\Sphoe}{\textphnc{\As}}
\newcommand{\hethphoe}{\textphnc{\Ahd}}


\DeclareRobustCommand*{\Sagelogo}{
  \begin{tikzpicture}[line width=.2ex,line cap=round,rounded corners=.01ex,baseline=-.2ex]
    \draw(0,0) -- (.75em,0)
      -- (.75em,.7ex) -- (.25em,.7ex)
      -- (.25em,.75\ht\strutbox) -- (.75em,.75\ht\strutbox);
    \draw(2em,0) -- (1.6em,0)
      -- (1.3em,.75\ht\strutbox) -- (.9em,.75\ht\strutbox)
      -- (.9em,0) -- (1.3em,0)
      -- (1.6em,.75\ht\strutbox) -- (2.1em,.75\ht\strutbox)
      -- (2.1em,-\dp\strutbox) -- (1.45em,-\dp\strutbox);
    \draw(3em,0) -- (2.25em,0)
      -- (2.25em,.75\ht\strutbox) -- (2.8em,.75\ht\strutbox)
      -- (2.8em,.7ex) -- (2.35em, .7ex);
  \end{tikzpicture}
}

\maketitle
\section*{Abstract}
In this article, we study the decomposition into irreducible components of the fixed point locus under the action of $\Gamma$ a finite subgroup of $\SL_2(\C)$ of the smooth Nakajima quiver variety of the Jordan quiver. The quiver variety associated with the Jordan quiver is either isomorphic to the punctual Hilbert scheme in $\C^2$ or to the Calogero-Moser space. We describe the irreducible components using quiver varieties of McKay's quiver associated with the finite subgroup $\Gamma$ and we have given a general combinatorial model of the indexing set of these irreducible components in terms of certain elements of the root lattice of the affine Lie algebra associated with $\Gamma$. Finally, we prove that for every projective, symplectic resolution of a wreath product singularity, there exists an irreducible component of the fixed point locus of the punctual Hilbert scheme in the plane that is isomorphic to the resolution.

\section{Introduction}

Fix $\Gamma$ a finite subgroup of $\SL_2(\C)$. The Jordan quiver is the quiver with one vertex and one arrow. This article is dedicated to the study of the $\Gamma$-fixed point locus of the smooth Jordan quiver variety. This quiver variety is isomorphic either to the Calogero-Moser space when the deformation parameter is nonzero or to the punctual Hilbert scheme of points in $\C^2$.
Iain Gordon \cite[Lemma $7.8$]{Gor08} has identified the irreducible components of the $\Gamma$-fixed point locus of the Hilbert scheme of point using quiver varieties of $Q_{\Gamma}$ when $\Gamma$ is of type $A$.
Moreover, Alexander Kirillov Jr. has studied the $\Gamma$-fixed point locus of the punctual Hilbert scheme using quiver varieties. Let us denote by $\Hk_n$ the Hilbert scheme of $n$ points in $\C^2$. He showed, in a concise proof and based on nonconstructive arguments \cite[Theorem $12.13$]{Kir06}, that there exists an isomorphism between the irreducible components of $\Hk_n^{\Gamma}$ and certain quiver varieties.
Finally, C\'edric Bonnaf\'e and Ruslan Maksimau \cite[Theorem $2.11$]{BM21} have described the irreducible components of $\Gamma$-fixed point loci of smooth generalized Calogero-Moser spaces using quiver varieties of $Q_{\Gamma}$ also when $\Gamma$ is of type $A$. Based on this work, the first result of this article is that the irreducible components of the $\Gamma$-fixed point locus of the Jordan quiver variety, denoted by $\QVJnb$, for $(\theta,\lb)\neq (0,0)$ (where $\theta$ is the stability parameter and $\lb$ the deformation parameter) are described in terms of quiver varieties of the McKay quiver of $\Gamma$. From now on, let us refer to quiver varieties associated with the McKay quivers as McKay quiver varieties. More precisely, we prove the following theorem in section \ref{chap_decomp}.\\
\begin{theorem1}
\label{thmI}
Let $(\theta,\lb) \in \mathbb{Q} \times \C$ such that $(\theta,\lb) \neq (0,0)$.
For each integer $n$ and each finite subgroup $\Gamma$ of $\SL_2(\C)$, the variety $\QVJnb^{\Gamma}$ decomposes into irreducible components
\begin{equation*}
\QVJnb^{\Gamma} = \coprod_{d \in \mathcal{A}_{n,\oando}}{\QV_d}
\end{equation*}
\noindent where $\QV_d$ is isomorphic to $\QV_{\theta}^{\lb\delta}(M^{\sigma})$ a McKay quiver variety.\\
\end{theorem1}

In the course of proving this Theorem, it became clear that a more intrinsic setting to work with McKay quiver varieties is needed, especially when $\Gamma$ is of type $D$ and $E$.  The development of this technical setting is inspired by the work of George Lusztig \cite[Section 2]{Lusz92}, of Michela Varagnolo and Eric Vasserot \cite[Section 2]{VV99}, as well as the one of Weiqiang Wang \cite[Theorem 5.1]{W00}.

 To be more precise, if one denotes by $\mathbf{Rep}$ the category of representations of the double framed McKay quiver and by $\mathbf{McK}$ a category defined with $\Gamma$-modules, then for every orientation $\Omega$ of the McKay quiver, two functors $F$ and $G$ between these categories are constructed in section \ref{chap_quiver}. We prove that these two functors define an equivalence of categories. This has already been done by Lusztig when the quiver is not framed \cite[Section 2.2]{Lusz92}. Furthermore, in the same section, we show that McKay quiver varieties denoted by $\QVTdef(d,d^f)$ and the analog of quiver varieties, denoted by $\mathcal{M}_{\theta}^{\lb}(M,M^f)$, coming naturally out of the category $\mathbf{McK}$ are isomorphic as algebraic varieties. Varagnolo and Vasserot work with the framed McKay quiver and with McKay quiver varieties but have not stated and proved this equivalence.

In the fourth section, we obtain a combinatorial description of the indexing set of these irreducible components in terms of elements of the root system of the Kac-Moody Lie algebra associated with $\Gamma$. Thanks to McKay's correspondence, one can associate to $\Gamma$ an affine Lie algebra. Fix a base of simple roots of this affine Lie algebra, associated with the fundamental Cartan chamber. Denote by $\mathcal{Q}$ the root lattice of this affine Lie algebra. The size of an element  $\alpha \in \mathcal{Q}$ is the average of the coefficients of $\alpha$ in a chosen set of simple roots weighted by the dimensions of the irreducible representations of $\Gamma$ (cf. Definition \ref{deff_size}). Generalizing \cite[Lemma $2.6$]{BM21} to type $D$ and $E$, another statistic on $\mathcal{Q}$ can be defined called the weight (cf. Definition \ref{deff_weight}). We prove the following theorem in section \ref{chap_irr}.\\
\begin{theorem2}
For each integer $n$, the indexing set of the irreducible components of $\Hk_n^{\Gamma}$ and of the irreducible components of the $\Gamma$-fixed point locus of the Calogero-Moser $n$-space are equal to the set of all positive linear combinations of simple roots of size $n$ and nonnegative weight.\\
\end{theorem2}
\noindent Finally in section \ref{sect_res}, we are interested in the classification of all projective, symplectic resolutions of the wreath product singularities \cite[Corollary $1.3$]{BC18}. If $k \geq 1$ is an integer, denote by $\Gamma_k$ the following wreath product $\Sk_k \ltimes \Gamma^k$. The wreath product singularity associated with $\Gamma$ and $k$ is then $\mathcal{Y}_k:=\Ctwo^k/\Gamma_k$. We show the following theorem.\\

\begin{theorem3}
For each projective, symplectic resolution $X \to \mathcal{Y}_k$, there exists $n \in \Z_{\geq 1}$ and there exists an irreducible component $\mathcal{I}$ of $\Hk_{n}^{\Gamma}$ such that $X \iso \mathcal{I}$ over $\mathcal{Y}_k$.
\end{theorem3}

\subsection{Notation for Kac-Moody Lie algebras}
\markboth{Notation}{Notation}
\noindent Here is a good place to recall common notation on Kac-Moody Lie algebras.
Take a finite or affine Kac-Moody Lie algebra $\mathfrak{g}$ and a Cartan subalgebra $\mathfrak{h}$ associated with the generalized Cartan matrix $A$. Let $\Pi$ be the set of simple roots associated with the fundamental Cartan chamber and $\Pi^{\vee}$ the set of simple coroots. Note that a realization (see \cite[paragraph 1.1]{kac90}) of $\transp{A}$ is given by the data $\left(\mathfrak{h}^*,\Pi^{\vee},\Pi\right)$. Denote by $\Dyn$ the Dynkin diagram associated with $\mathfrak{g}$, by $\Phi^+ \subset \Phi\subset \mathfrak{h}^*$ the positive roots and root system associated with the data $(\mathfrak{g},\mathfrak{h})$ and by $\Phi^{\vee} \subset \mathfrak{h}$ the coroot system. Denote by $\mathcal{Q}$ and $\mathcal{Q}^{\vee}$ respectively the root and coroot lattice. Let $W$ be the Weyl group associated with $(\mathfrak{h}, \Pi, \Pi^{\vee})$. Let the natural pairing between $\mathfrak{h}^*$ and $\mathfrak{h}$ be denoted by $\langle,\rangle$.

If $\mathfrak{g}$ is of affine type, let $\alpha_0$ and $\alpha^{\vee}_0$ be respectively the root and coroot defined in \cite[paragraph 6.4]{kac90}. Moreover let  $\mathbf{d}$ be the element of $\mathfrak{h}$ defined in \cite[paragraph 6.2]{kac90}. In that case, note that $\Pi^{\vee} \cup \{\mathbf{d}\}$ is a basis of $\mathfrak{h}$. In the affine case, one root and one coroot will have a special role. Let us denote by $\delta:=\sum_{\alpha \in \Pi}{\delta_{\alpha}\alpha} \in \Phi^+$ the null root, which is such that $A\delta=0$ and such that the integral coefficients $(\delta_{\alpha})_{\alpha \in \Pi} \in \mathbb{Z}_{\geq 0}$ are the smallest possible. In the same way, one can define $\delta^{\vee}:=\sum_{\alpha^{\vee} \in \Pi^{\vee}}{\delta^{\vee}_{\alpha^{\vee}}\alpha^{\vee}} \in \Phi^{\vee}$ the null coroot, which is such that $\transp{A}\delta^{\vee}=0$ and the integral coefficients $(\delta^{\vee}_{\alpha^{\vee}})_{\alpha^{\vee} \in \Pi^{\vee}} \in \mathbb{Z}_{\geq 0}$ are the smallest possible.

Finally, the fundamental weights and coweights will be needed. For each $\alpha \in \Pi$, consider $\Lambda_{\alpha} \in \mathfrak{h}^*$ such that:
\begin{equation*}
\forall \alpha^{\vee} \in \Pi^{\vee},\langle \Lambda_{\alpha},\alpha^{\vee} \rangle =
\begin{cases}
1 & \text { if } \langle \alpha, \alpha^{\vee} \rangle = 2\\
0 & \text { else}
\end{cases}
\text{ and } \langle \Lambda_{\alpha}, \mathbf{d}\rangle = 0.
\end{equation*}
Similarly, one can define the fundamental coweights. For each $\alpha^{\vee} \in \Pi^{\vee}$, consider $\Lambda^{\vee}_{\alpha^{\vee}} \in \mathfrak{h}$ such that:
\begin{equation*}
\forall \alpha \in \Pi,\langle \alpha, \Lambda^{\vee}_{\alpha^{\vee}}\rangle =
\begin{cases}
1 & \text { if } \langle \alpha, \alpha^{\vee} \rangle = 2\\
0 & \text { else}
\end{cases} \text{ and } \langle \Lambda_{\alpha_0}, \Lambda^{\vee}_{\alpha^{\vee}} \rangle = 0.
\end{equation*}
Let us denote by $\mathcal{P}$ and $\mathcal{P}^{\vee}$ respectively the weight and coweight lattice respectively generated by the fundamental weights and fundamental coweights.\\

\begin{deff}
\label{deff_size}
If $\alpha \in \mathcal{Q}$, then the size of $\alpha$ is defined as
\begin{equation*}
|\alpha|:=\langle \alpha,\sum_{\alpha^{\vee} \in \Pi^{\vee}}{\delta^{\vee}_{\alpha^{\vee}}\Lambda^{\vee}_{\alpha^{\vee}}}\rangle
\end{equation*}
 and if $\alpha^{\vee} \in \mathcal{Q}^{\vee}$, define:
\begin{equation*}
|\alpha^{\vee}|:=\langle \sum_{\alpha \in \Pi}{\delta_{\alpha}\Lambda_{\alpha}}, \alpha^{\vee} \rangle.
\end{equation*}
\end{deff}

\noindent In the affine case, the group $W$ is the affine Weyl group of the associated finite type. Let us denote respectively by $W_{\mathrm{fin}}$, $\mathcal{Q}_{\mathrm{fin}}$ and $\Phi_{\mathrm{fin}}$ the Weyl group, root lattice and set of roots of the associated finite type. One has the following isomorphism (\cite[Chap. VI, paragraph 2, Prop. $1$]{Bourblie})
\begin{equation*}
 W_{\mathrm{fin}} \ltimes \mathcal{Q}_{\mathrm{fin}} \iso W
\end{equation*}
For $a\in \mathcal{Q}_{\mathrm{fin}}$, we will denote by $t_a \in W$ the corresponding element.\\
All schemes will be over $\mathbb{C}$ and we will also suppose that the structure morphism is separated and of finite type. An algebraic variety will be an integral scheme.

\section{Quiver variety over McKay quivers}
\label{chap_quiver}

This section is here to set up the quiver theoretical background needed in section \ref{chap_decomp}. The subsection \ref{McKay quivers} is inspired by the work of George Lusztig \cite[Section 2]{Lusz92}, of Michela Varagnolo and Eric Vasserot \cite[Section 2]{VV99}, as well as the one of Weiqiang Wang \cite[Theorem 5.1]{W00}.
This section is decomposed into four parts. Firstly, we recall the construction of the representation space of a double framed quiver and of its natural symplectic structure. Then, if one starts with the double framed McKay quiver of $\Gamma$, one can construct the representation space of this quiver in terms of $\Gamma$-modules. This is the new setting to work with McKay quiver varieties. This setting will be used to prove the main theorem in section \ref{chap_decomp}. In the third part, we show that over the double framed McKay quiver the two representation spaces are isomorphic and that moreover, the symplectic forms coincide as well as the momentum maps. Moreover, we will, for McKay quivers, define an analog of Nakajima's quiver varieties that does not involve the choice of an orientation of the underlying quiver and thus will provide a more intrinsic point of view. We will show that this variety is isomorphic to Nakajima's original quiver variety. Finally, we will give a couple of results that are necessary for the following sections.

\subsection{Representations of double framed quivers}
\label{section_symp}
Let us begin this section with general notation for quivers.
Take an undirected multigraph $G:=(I_G,E_G)$ where $I_G$ is the set of vertices and $E_G$ the multiset of undirected edges and choose an orientation $\Omega:E_G \rightarrow I_G \times I_{G}$ for $G$. When the undirected multigraph and the orientation are clear from context, we shorten the notation and forget the dependencies in $G$ and $\Omega$, to improve readability.

The framed undirected multigraph associated with $G$ denoted by $G^f$ is the multigraph with one extra vertex for each vertex in $I_G$ so that $I_{G^f}=I_G\coprod I_G$ is the set of vertices and the multiset of edges is defined as $E_{G^f}:= E_G \coprod \big\{\{i,j_i\}| i \in I_G \big\}$ where $j_i$ denotes the new vertex associated with  $i \in I_G$.  Let us denote by $Q_G(\Omega)$ the quiver associated with the undirected multigraph $G$ and the orientation $\Omega$.  The source and target maps of $Q$ are given by $\Omega$ and will be respectively denoted by $h'$ and $h''$ for $h \in E$. Let us impose the convention that the orientation for the framing arrows is the  "outgoing" orientation. This implies that choosing an orientation of a framed quiver $\Omega^f:E_{G^f} \to I_{G^f} \times I_{G^f}$ is the same data as choosing an orientation $\Omega$ of the associated unframed undirected multigraph.
Consider the following orientation\\
\begin{center}
$\Omega^{\mathrm{op}}: \begin{array}{ccc}
E_G & \to & I_G \times I_G\\
h& \mapsto & (h'',h')
\end{array}$.\\
\end{center} \vspace*{0.25cm}
The quiver $Q_G(\Omega^{\op})$ is called the opposite quiver of $Q$ i.e. the quiver which has the same underlying undirected multigraph $G$ but reversed orientation. The multiset of arrows of $Q_G(\Omega^{\op})$ will be denoted by $E^{-}$. Finally denote by  $\overline{Q}$ the double quiver associated with $G$ which has the same set of vertices as $G$ and the multiset of arrows ${\overline{E}:=E \coprod E^{-}}$ so that $\overline{Q}$ does not depend on the choice of $\Omega$. Let $\bar{h}$ be the involution of $\overline{E}$ sending an edge $h \in \overline{E}$ to its reversed edge.

Let $\Delta_G$ be the free abelian group associated with the set $I$. Let  ${\Delta_G^+ \subset \Delta_G}$ be the free monoid associated with $I$. The set $\Delta_G^+$ will be referred to as the set of dimension parameters. A dimension parameter $d$ will often be defined by giving nonnegative integers $(d_{\nu})_{\nu \in I}$ which defines $d:=\sum_{\nu \in I}{d_{\nu} \nu} \in \Delta^+$. Let us moreover consider the set of stability parameters denoted by $\Theta:=\Hom_{\mathbb{Z}}(\Delta,\mathbb{Q})$ and the set of deformation parameters denoted by $\Lambda:=\Hom_{\mathbb{Z}}(\Delta,\C)$. If $\theta \in \mathbb{Q}$, let us denote by $\theta$ the stability parameter defined as $\theta(\nu):=\theta$ for all $\nu \in I$. If $\lb \in \C$, the same notation will be used for the constant  deformation parameter equal to $\lb$. We will also use the following notation:
\begin{align*}
\centering
\bullet \text{ }\Theta^+:=\{ \theta \in \Theta &\mid \forall \nu \in I, \theta(\nu) \geq 0\},\\
\bullet \text{ }\Theta^{++}:=\{ \theta \in \Theta &\mid \forall \nu \in I, \theta(\nu) >0\},\\
\bullet \text{ }\Lambda^{+}:=\{ \lambda \in \Lambda &\mid \forall \nu \in I, \lb(\nu) \geq0\},\\
\bullet \text{ }\Lambda^{++}:=\{ \lambda \in \Lambda &\mid \forall \nu \in I, \lb(\nu) >0\}.
\end{align*}
Let us denote $\mathbf{Rep}$ the category of representations of the quiver $\DMQGf$ in which objects are tuples $(V,V^f,(x_{h})_{h\in \overline{E_{G}}}, (v^1_{i},v^2_i)_{i \in I_G})$ where
\begin{itemize}
\item $V=\oplus_{i \in I_{G}}{V_{i}}$ is an $I_G$-graded complex vector space.
\item $V^f=\oplus_{i \in I_{G}}{V^f_{i}}$ is an $I_G$-graded complex vector space.
\item $\forall h \in \overline{E_G}, x_{h}:V_{h'} \to V_{h''}$ is a linear map.
\item $\forall i \in I_G, v_{i}^1:V^f_{i}\to V_{i}$  is a linear map.
\item $\forall i \in I_G, v_{i}^2: V_{i} \to V^f_{i}$ is a linear map.
\end{itemize}
Note that $v_i^1$ and $v_i^2$ correspond to the representation of the framing. A morphism between $(V,V^f,x,v^1,v^2)$ and $(\tilde{V},\tilde{V}^f,\tilde{x},\tilde{v}^1,\tilde{v}^2)$ is given by the following data. For all $i \in I$, we need to provide a morphism $\phi_{i}: V_{i}\to \tilde{V}_{i}$ and a morphism $\psi_{i}: V^f_{i} \to \tilde{V}^f_i$ such that for all $h \in \overline{E}$ and $i \in I$ the following diagrams commute
\begin{center}
\begin{tikzcd}[column sep=huge]
V_{h'} \ar[r, "\phi_{h'}"] \ar[d, "x_h"'] & \tilde{V}_{h'}\ar[d, "\tilde{x}_{h}" ]\\
V_{h''} \ar[r,"\phi_{h''}"']  & \tilde{V}_{h''}
\end{tikzcd}
\end{center}
\begin{center}
\begin{tikzcd}[column sep=huge]
V_{i} \ar[r, "v^2_{i}"] \ar[d, "\phi_{i}"'] & V^f_{i} \ar[r, "v^1_{i}"] \ar[d, "\psi_{i}" ] & V_{i} \ar[d, "\phi_{i}"]\\
\tilde{V}_{i} \ar[r,"\tilde{v}^2_{i}"']  & \tilde{V}^f_{i} \ar[r, "\tilde{v}^1_{i}"'] & \tilde{V}_{i}
\end{tikzcd}.
\end{center}

\begin{deff}
Given two $I$-graded complex vector spaces $V,V^f$ define the representation space of $\DMQGf$ with fixed $I$-graded complex vector spaces $V,V^f$ as
\begin{equation*}
\RepQ^{G}_{V, V^f}:= \bigoplus_{h \in \overline{E}}{\Hom(V_{h'},V_{h''}}) \oplus \bigoplus_{i \in I}{\Hom(V^f_{i}, V_{i})} \oplus \bigoplus_{i \in I}{\Hom(V_{i}, V^f_{i})}
\end{equation*}
\end{deff}

\begin{rmq}
We have the following remark concerning this notation.
\begin{enumerate}[label=(\roman*)]
\item If $V$ and $V^f$ are clear from context, this representation space will just  be denoted by $\mathcal{R}$.
\item Take a vertex $i_0 \in I$ and define $V^{f,i_0}$ such that $V^{f,i_0}_i$ is equal to $\C$ in degree $i_0$ and $0$ anywhere else. In the following, let us shorten $\RepQ_{V,V^{f,i_0}}^{G}$ to just $\RepQ_{V,i_0}^{G}$.  Note that in that case $v^1$ is just a linear map from $\C$ to $V_{i_0}$ i.e. can be identified with an element of $V_{i_0}$. We will refer to $\RepQ_{V,i_0}^G$ as the representation space of the double, framed at $i_0$, quiver $Q_G(\Omega)$.
\end{enumerate}
 \end{rmq}
\vspace*{0.25cm}

\noindent The construction of a symplectic structure on $\mathcal{R}$ will depend on the choice of an orientation $\Omega$. Recall that we have chosen an orientation $\Omega$ of $G$. The sign function associated with $\Omega$ will be denoted by $\varepsilon: \overline{E} \to \{-1,1\}$ . This map is such that ${\varepsilon(h):=1}$ for all $h \in E^+$ and $\varepsilon(h):=-1$ for all $h \in E^{-}$. Define the symplectic form
\begin{equation*}
\omega_{\varepsilon}:
\begin{array}{ccc}
\mathcal{R} \times \mathcal{R} & \to & \C \\
((x,v^1,v^2),(\tilde{x},\tilde{v}^1,\tilde{v}^2)) & \mapsto & \sum_{h \in \overline{E}}{\varepsilon(h)\mathrm{Tr}(x_h \tilde{x}_{\bar{h}})} + \sum_{i \in I}{\mathrm{Tr}(v^1_{i}\tilde{v}^2_{i} -\tilde{v}^1_{i}v^2_{i})}  \\
\end{array}
\end{equation*}
This structure comes from the natural isomorphism with the cotangent bundle of the representation space of the framed quiver.
Consider the group ${\G(V):=\prod_{i \in I_{G}}{\GL(V_{i})}}$ and its action on $\mathcal{R}$ given by
\begin{equation*}
g.(x,v^1,v^2):=\left(g.x,(g_{i}v^1_{i})_{i \in I},(v^2_{i} g_{i}^{-1})_{i \in I}\right)
\end{equation*}
for $g \in \G(V), (x,v^1,v^2) \in \mathcal{R}$ and where $\forall h \in \overline{E}, (g.x)_h:=g_{h''}x_hg_{h'}^{-1}$.\\

\begin{rmq}
Note that the main difference between the framed and unframed setting is that here we let $\G(V)$ act and not $\G(V)\times \G(V^f)$.
Moreover, with this action, the symplectic form $\omega_\varepsilon$ is $\G(V)$-invariant.
\end{rmq}

\noindent For $\theta \in \Theta_G$ such that $\mathrm{Im}(\theta) \subset \mathbb{Z}$, define the character of $\G(V)$ associated with $\theta$ as
\begin{center}
$\chi_{\theta}: \begin{array}{ccccc}
\G(V) &\to & \C^*\\
(g_i)_{i \in I} & \mapsto & \prod_{i \in I}{\mathrm{det}(g_i)^{\theta(i)}}\\
\end{array}$.
\end{center}
\vspace*{0.25cm}

\noindent Consider now the momentum map attached to that action \cite[Definition 9.43]{Kir06}
\begin{equation*}
\mu_\varepsilon:
\begin{array}{ccc}
\mathcal{R} & \to & \bigoplus_{i \in I}{\mathrm{End}(V_{i})}\\
(x,v^1,v^2) & \mapsto & \sum_{h \in \overline{E}, h''=i}{\varepsilon(h)x_hx_{\bar{h}}} + \sum_{j \in I}{v^1_{j}v^2_{j}}  \\
\end{array}.
\end{equation*}
For $\lb \in \Lambda$, let us still denote by $\lb:=\sum_{i \in I}{\lb(i)\id_{V_i}} \in \bigoplus_{i \in I}{\mathrm{End}(V_{i})}$.

\subsection{Representation space of double framed McKay quivers}
\label{McKay quivers}
Take $k \in \mathbb{Z}_{\geq 0}$ and recall that we have fixed $\Gamma$ a finite subgroup of $\mathrm{SL}_2(\C)$. By a $\Gamma$-module, we mean a $\C[\Gamma]$-module of finite dimension.
This subsection is devoted to exposing the representation space of the double framed McKay quiver of $\Gamma$ in terms of $\Gamma$-modules.

Denote by $\RepGk:=\mathrm{Hom}_{\mathrm{gp}}(\Gamma, \GL_k(\C))$ and by $\CharGk$ the set of all characters of elements in $\RepGk$. Consider $\RepG := \bigcup_{k \geq 0}{\RepGk}$ and  $\CharG$ the set of all characters of representations in $\RepG$. For $\chi \in \CharG$, let us denote $k_{\chi}:=\chi(\id)$ and choose $\rho_{\chi} \in \mathrm{Rep}_{k_{\chi}}$ such that the character of $\rho_{\chi}$ denoted by $\mathrm{Tr}(\rho_{\chi})$ is $\chi$. Note that $\rho_{\chi}$ is determined, up to conjugation by an element of $\mathrm{GL}_{k_{\chi}}(\C)$. Denote the representation space of $\rho_{\chi}$ by $X_{\chi}$ which is endowed with the $\Gamma$-action given by $\rho_{\chi}$. Let $\IG$ be the set of all characters of irreducible representations of $\Gamma$. It is finite since $\Gamma$ is finite.

Denote by $\chi_0 \in \mathrm{Irr}$ the trivial character. The group $\Gamma$ being a subgroup of $\SL_2(\C)$, it has a natural representation of dimension $2$ called the standard representation and denoted by  $\rho_{\mathrm{std}}$. In the following, the character of the standard representation, which is irreducible whenever $\Gamma$ is not a cyclic group, and its associated representation space will be respectively denoted by $\chi_{\mathrm{std}}$ and $X_{\mathrm{std}}$.
If $A$ and $B$ are two $\Gamma$-modules, then $A \to_{\Gamma} B$ denotes a $\Gamma$-equivariant morphism and $\Hom_{\Gamma}(A,B)$ will denote the set of all $\Gamma$-equivariant morphisms between $A$ and $B$.\\

\noindent Consider $\mathbf{McK}$ the category in which objects are tuples of the form  $(M,M^f,\Delta,Z)$ where
\begin{itemize}
\item $M$ is a $\Gamma$-module
\item $M^f$ is a $\Gamma$-module
\item $\Delta: \Xstd \otimes M \to_{\Gamma} M$
\item  $Z_1:M^f \to_{\Gamma} M$
\item $Z_2:M \to_{\Gamma} M^f$
\end{itemize}
Morphisms of this category, between $(M,M^f,\Delta,Z_1,Z_2)$ and $(\tilde{M}, \tilde{M}^f, \tilde{\Delta}, \tilde{Z}_1, \tilde{Z}_2)$, are pairs $(\Phi, \Psi)$ where $\Phi: M \to_{\Gamma} \tilde{M}$  and $\Psi: M^f \to_{\Gamma} \tilde{M}^f$ are such that the following diagrams commute
\begin{center}
\begin{tikzcd}[column sep=huge]
\Xstd \otimes M \ar[r, "\Delta"] \ar[d, "\mathrm{id} \otimes \Phi"'] & M\ar[d, "\Phi" ]\\
\Xstd \otimes \tilde{M} \ar[r,"\tilde{\Delta}"']  & \tilde{M}
\end{tikzcd},
\end{center}

\begin{center}
\begin{tikzcd}[column sep=huge]
M \ar[r, "Z_2"] \ar[d, "\Phi"'] & M^f \ar[d, "\Psi" ] \ar[r, "Z_1"] & M \ar[d, "\Phi"]\\
\tilde{M} \ar[r,"\tilde{Z}_2"']  & \tilde{M}^f \ar[r, "\tilde{Z}_1"'] & \tilde{M}
\end{tikzcd}.
\end{center}
The universal property of the tensor product gives
\begin{equation*}
\Hom(\Xstd \otimes M, M) \simeq \Hom(\Xstd,\mathrm{End}(M)).
\end{equation*}
Moreover, if we let $\Gamma$ act on $\mathrm{End}(M)$ by conjugacy, then we have
\begin{equation*}
\Hom_{\Gamma}(\Xstd \otimes M, M) \simeq \Hom_{\Gamma}(\Xstd,\mathrm{End}(M)).
\end{equation*}

\noindent For  $\Delta \in \Hom_{\Gamma}(X_{\mathrm{std}} \otimes M, M)$ and $x \in \Xstd$, let us denote $\Delta_x$ the obtained endomorphism of $M$. Concretely, for all $m \in M$, $\Delta_x(m)$ is equal to $\Delta(x\otimes m)$.\\

\begin{deff}
Take two $\Gamma$-modules $M,M^f$ and define the following complex vector space attached to $M,M^f$
\begin{equation*}
\RepQ_{M,M^f} := \Hom_{\Gamma}(\Xstd \otimes M ,M) \oplus \Hom_{\Gamma}(M^f,M) \oplus \Hom_{\Gamma}(M,M^f).
\end{equation*}
\end{deff}

\begin{rmq}
When $M$ and $M^f$ are clear from context, we will just denote this representation space by $\RepQ$.\\
\end{rmq}

\begin{deff}
Let $\Aut_{\Gamma}(M)$ be the algebraic group of linear automorphisms of $M$ commuting with the $\Gamma$-action.
\end{deff}

\noindent We can now consider an action of $\Aut_{\Gamma}(M)$ on $\RepQ$ given by
\begin{equation*}
g.(\Delta,Z_1,Z_2):=(g.\Delta,gZ_1,Z_2g^{-1})
\end{equation*}
for $g \in \Aut_{\Gamma}(M)$, $(\Delta,Z_1,Z_2) \in \RepQ$ and where
\begin{equation*}
g.\Delta(x\otimes m):=g\Delta(x\otimes g^{-1}m), \quad \forall (x,m) \in \Xstd\times M.
\end{equation*}

\begin{rmq}
\label{rmq_inv}
Let us denote by $(e_1, e_2)$ the canonical basis of $\Xstd$.
\begin{enumerate}[label=(\roman*)]

\item One can check that $\Delta_{e_1}\Delta_{e_2}-\Delta_{e_2}\Delta_{e_1} \in \mathrm{End}_{\Gamma}(M)$, using the following fact
\begin{equation*}
\Delta_{g.e_1}\Delta_{g.e_2}-\Delta_{g.e_2}\Delta_{g.e_1}=\det(g)(\Delta_{e_1}\Delta_{e_2}-\Delta_{e_2}\Delta_{e_1}), \qquad \forall g \in \GL_2(\C).
\end{equation*}

\item In the special case where $M^f=X_{\chi}$ for some $\chi \in \Irr$, let us denote $\RepQ_{M,M^f}$ by $\RepQ_{M,\chi}$ and $\RepQ_{M}:=\RepQ_{M,\chi_0}$.\\
\end{enumerate}
\end{rmq}

\noindent Define an $\Aut_{\Gamma}(M)$-invariant symplectic form on $\RepQ_{M,M^f}$
\begin{equation*}
\omega:
\begin{array}{ccc}
\RepQ \times \RepQ & \to & \C\\
\left(\left(\Delta,Z_1,Z_2\right),\left(\tilde{\Delta},\tilde{Z}_1,\tilde{Z}_2\right)\right)& \mapsto & \Tr\left(\Delta_{e_1}\tilde{\Delta}_{e_2}-\Delta_{e_2}\tilde{\Delta}_{e_1}\right) + \Tr\left(Z_1\tilde{Z}_2-\tilde{Z}_1Z_2\right).  \\
\end{array}
\end{equation*}

\begin{rmq}
Note that $\omega$ is independent of the choice of basis of $\Xstd$ as long as one picks a basis $(e,f)$ such that $\det(e,f)=1$.
\end{rmq}

\noindent The momentum map attached to $\omega$ and the action of $\Aut_{\Gamma}(M)$ on $\RepQ$ is
\begin{equation*}
\mu:
\begin{array}{ccc}
\RepQ & \to & \mathrm{End}_{\Gamma}(M)\\
(\Delta,Z_1,Z_2) & \mapsto & \Delta_{e_1}\Delta_{e_2}-\Delta_{e_2}\Delta_{e_1} + Z_1Z_2 \\
\end{array}
\end{equation*}

\begin{rmq}
Note that the symplectic form and the momentum map only depend on the finite group $\Gamma$.\\
\end{rmq}

\begin{deff}
Define the McKay undirected multigraph $G_{\Gamma}$ associated with $\Gamma$ in the following way. The set of vertices is $\IG$ the set of all irreducible characters of $\Gamma$. There is an edge between a pair of irreducible characters $(\chi,\chi')$ if and only if $\langle \chi\chi_{\mathrm{std}}|\chi'\rangle \neq 0$ with multiplicity $\langle \chi\chi_{\mathrm{std}}|\chi'\rangle$.\\
\end{deff}
\begin{rmq}
Note that $G_{\Gamma}$ is indeed undirected because $\Gamma$ is a subgroup of $\mathrm{SL}_2(\C)$. By choosing an orientation of the McKay multigraph of $\Gamma$, we define the McKay quiver denoted by $Q_{\Gamma}$. The choice of orientation will not be relevant to this subsection.\\
\end{rmq}
\noindent  Let us finish this subsection by presenting a way to formulate the McKay Correspondence. This will be useful in section \ref{chap_irr}.\\

\begin{prop}[McKay Correspondence, \cite{mckay80}]
\label{mckay}
If $V$ denotes the vector space $\Delta_{G_\Gamma}\otimes \C$, then $\Irr$ defines an affine root system in $V$ of type $\Gamma$.\\
\end{prop}

\begin{rmq}
\label{rmq_mckay}
To be a bit more precise and use the language developed by Kac in \cite[Chapter $1$]{kac90}. Let $\tilde{\Pi}^{\vee} := \Irr$, which is then a base of $V$. We can now construct a realization of the generalized Cartan matrix of type $\Gamma$ out of $V$ and $\tilde{\Pi}^{\vee}$. Consider the following linear map defined on the base $\Irr$:
\begin{center}
$\phi: \begin{array}{ccc}
V & \to & V^*\\
\chi & \mapsto & (\psi \mapsto 2\langle \chi, \psi \rangle - \langle \chi \chi_{\mathrm{std}}|\psi \rangle) \\
\end{array}$.\\
\end{center}

\noindent Note that the element $\phi(\chi) \in V^*$ is defined on the base $\Irr$ of $V$. By construction, we have that $\phi(\chi)(\psi)$ is exactly the corresponding coefficient of the generalized Cartan matrix of type $\Gamma$. Let us denote this coefficient by $a_{\chi,\psi}$. Take $\mathfrak{t}_0$ to be a one-dimensional complex vector space. Consider $\mathfrak{t} := \mathfrak{t}_0 \oplus V$. For each $\chi \in \Irr$, define $\tilde{\alpha}_{\chi} \in \mathfrak{t}^*$ as follows:
\begin{equation*}
\tilde{\alpha}_{\chi}(t_0+\sum_{\psi \in \Irr}{c_{\psi}\psi}) := t_0 + \sum_{\psi \in \Irr}{c_{\psi}a_{\psi,\chi}}, \quad \forall (t_0,(c_{\psi})_{\psi}) \in \mathfrak{t}_0 \times \C^{|\Irr|}.
\end{equation*}
The set $\tilde{\Pi}:=\{\tilde{\alpha}_{\chi} \mid \chi \in \Irr\}$ is linearly independant. The triple $(\mathfrak{t}, \tilde{\Pi}, \tilde{\Pi}^{\vee})$ is a realization of the generalized Cartan matrix of type $\Gamma$. Consider  ${(\mathfrak{h},\Pi,\Pi^{\vee}):=(\mathfrak{t}^*,\tilde{\Pi}^{{\vee}^{**}},\tilde{\Pi})}$ where
\begin{center}
$^{**}:\begin{array}{ccc}
\mathfrak{t} & \to & \mathfrak{h}^*=\mathfrak{t}^{**}\\
t & \mapsto & \left(\phi \mapsto \phi(t)\right)\\
\end{array}$.
\end{center}
For $\chi \in \Irr$, let us denote $\alpha_{\chi}:=\chi^{**} \in \Pi$.
The type $\Gamma$ being simply laced, $(\mathfrak{h}, \Pi, \Pi^{\vee})$ is also a realization of the generalized Cartan matrix of type $\Gamma$ and we will referred to it as the realization given by the McKay correspondence.
\end{rmq}

\subsection{Equivalence of categories between $\mathbf{Rep}$ and $\mathbf{McK}$}

In this section, we will establish an equivalence between the two previously defined categories. To be more precise, let us show that $\mathbf{Rep}(\DMQf)$ and $\mathbf{McK}$ are equivalent and that this equivalence is compatible with the momentum maps.
To do so, we have to make choices. Thanks to the fact that $\Gamma$ is a finite subgroup of $\SL_2(\C)$, we know that for all $h \in \overline{E}$, ${T_{h}:= \Hom_{\Gamma}(\Xstd \otimes X_{h'}, X_{h''})}$ is of dimension either $0,1$ or $2$ (it is only of dimension $2$ when $\Gamma=\mu_2$, the cyclic group with two elements).

For each edge ${h \in \overline{E}}$, choose a nonzero element $y_{h}^0 \in \Hom_{\Gamma}(\Xstd \otimes X_{h'}, X_{h''})$. The irreducibility of $X_{h''}$ implies that $y_h^0$ is surjective. Since representations of finite groups are semisimple, we can consider a $\Gamma$-equivariant section of $y_h^0$ denoted by $\tilde{y}_h^0 \in \Hom_{\Gamma}(X_{h''},\Xstd\otimes X_{h'})$. For $\mu_2$, we can construct everything explicitly. Take the following labeling of the double, framed at $\chi_0$ McKay quiver of $\mu_2$
\begin{figure}[H]
\centering
\includegraphics[scale=0.25]{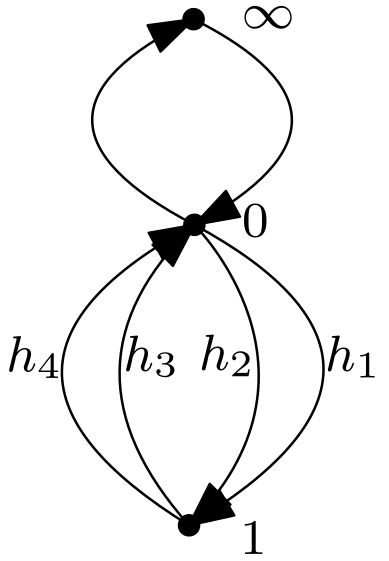}
\end{figure}

\noindent such that $\overline{h_1}=h_4$ and $\overline{h_3}=h_2$. Let $\chi_1$ denote the nontrivial irreducible character of $\mu_2$ and let $X_1:=X_{\chi_1}$. With this notation, $X_{\mathrm{std}}=X_1\oplus X_1$. Take these maps
\begin{multicols}{2}
\begin{itemize}
\item $y_{h_1}^0: \begin{array}{ccc}
(X_1 \oplus X_1) \otimes X_0 & \to & X_1\\
(a,b)\otimes 1 & \mapsto & a \\
\end{array}$

\item $y_{h_2}^0: \begin{array}{ccc}
(X_1 \oplus X_1) \otimes X_0 & \to & X_1\\
(a,b)\otimes 1 & \mapsto & b \\
\end{array}$

\item $y_{h_3}^0: \begin{array}{ccc}
(X_1 \oplus X_1) \otimes X_1 & \to & X_0\\
(a,b)\otimes c & \mapsto & ac \\
\end{array}$

\item $y_{h_4}^0: \begin{array}{ccc}
(X_1 \oplus X_1) \otimes X_1 & \to & X_0\\
(a,b)\otimes c & \mapsto & bc \\
\end{array}$

\item $\tilde{y}_{h_1}^0: \begin{array}{ccc}
X_1 & \to & (X_1\oplus X_1) \otimes X_0\\
1 & \mapsto & (1,0)\otimes 1\\
\end{array}$

\item $\tilde{y}_{h_2}^0: \begin{array}{ccc}
X_1 & \to & (X_1\oplus X_1) \otimes X_0\\
1 & \mapsto & (0,1)\otimes 1\\
\end{array}$

\item $\tilde{y}_{h_3}^0: \begin{array}{ccc}
X_0 & \to & (X_1\oplus X_1) \otimes X_1\\
1 & \mapsto & (1,0)\otimes 1\\
\end{array}$

\item $\tilde{y}_{h_4}^0: \begin{array}{ccc}
X_0 & \to & (X_1\oplus X_1) \otimes X_1\\
1 & \mapsto & (0,1)\otimes 1\\
\end{array}$.
\end{itemize}
\end{multicols}

\begin{lemme}
\label{sum works}
For each $\chi \in \Irr$, $\sum_{h \in \overline{E}, h'=\chi}{\tilde{y}_h^0\circ y_h^0}=\mathrm{id}_{\Xstd\otimes X_{\chi}}$.
\end{lemme}
\begin{proof}
For $\Gamma=\mu_2$ it is an easy computation. Now take $\Gamma\neq \mu_2$.
It is then  clear that $\sum_{h \in \overline{E}, h'=\chi}{y_h^0} \in \Hom_{\Gamma}(\Xstd \otimes X_{\chi}, \bigoplus_{h \in \overline{E},h'=\chi}{X_{h''}})$ is surjective. By construction of the McKay graph, we have
\begin{equation*}
\mathrm{dim}(\Xstd\otimes X_{\chi})=\sum_{h \in \overline{E}}{\mathrm{dim}(X_{h''})}.
\end{equation*}
In particular, this implies that $\sum_{h \in \overline{E}, h'=\chi}{y_h^0}$ is an isomorphism. It is then enough to prove that
\begin{equation*}
\left(\sum_{l \in \overline{E}, l'=\chi}{y_l^0} \right) \circ \left(\sum_{h \in \overline{E}, h'=\chi}{\tilde{y}_h^0\circ y_h^0} \right) = \sum_{h \in \overline{E}, h'=\chi}{y_h^0}.
\end{equation*}
For each pair of edge $(l,h) \in \overline{E}^2$ such that $l'=h'$, using Schur's Lemma, we have that
\begin{equation*}
y_l^0\circ \tilde{y}_h^0 =
\begin{cases}
\mathrm{id}_{X_{h''}} & \text{ if } l=h,\\
0 & \text{ otherwise}.
\end{cases}
\end{equation*}
\end{proof}

\noindent  Consider a special dimension parameter $\delta \in \Delta_{\Gamma}$ defined as $\delta_{\chi}:=\mathrm{dim}(X_{\chi})$ for all $\chi \in \Irr$. This notation does not conflict with the notation of the null root (see end of the subsection \ref{imp_res}). With the preceding choice of orientation and maps for $\mu_2$, everything that will follow will also work.\\
\noindent Define two functors $\mathcal{F}: \mathbf{Rep} \to  \mathbf{McK}$ and $\mathcal{G}: \RepGa  \to  \mathbf{Rep}$.\\
First, let
\begin{equation*}
 \mathcal{F}:
\begin{array}{ccc}
\mathbf{Rep} & \to & \mathbf{McK}\\
(V,V^f,x,v^1,v^2) & \rightsquigarrow &  (M, M^f, \Delta, Z_1, Z_2)\\
\end{array}
\end{equation*}
where
\begin{itemize}
\item $M:=\bigoplus_{\chi \in \Irr}{V_{\chi} \otimes X_{\chi}}$,
\item $M^f:=\bigoplus_{\chi \in \Irr}{V^f_{\chi} \otimes X_{\chi}}$,
\item $\Delta$ is defined as the following composition
\begin{center}
$\begin{tikzcd}
\Xstd \otimes \left( \bigoplus_{\chi \in \Irr}{V_{\chi} \otimes X_{\chi}} \right) \ar[d, "\rotatebox{90}{\(\sim\)}"'] \ar[r, "\Delta"] & \bigoplus_{\chi' \in \Irr}{V_{\chi'} \otimes X_{\chi'}}\\
\bigoplus_{\chi \in \Irr}{\Xstd \otimes V_{\chi} \otimes X_{\chi}} \ar[ur, "\sum_{h \in \overline{E}}{y_h^0 \otimes x_h}"']
\end{tikzcd}$,
\end{center}

\item $Z_1 := (\sum_{\chi \in \Irr}{v^1_{\chi}\otimes \frac{1}{\delta_{\chi}}\mathrm{id}_{X_{\chi}}}): \bigoplus_{\chi \in \Irr}{ V^f_{\chi} \otimes X_{\chi}} \to \bigoplus_{\chi \in \Irr}{V_{\chi} \otimes X_{\chi}}$,
\item $Z_2 := (\sum_{\chi \in \Irr}{v^2_{\chi}\otimes \mathrm{id}_{X_{\chi}}}): \bigoplus_{\chi \in \Irr}{ V_{\chi} \otimes X_{\chi}} \to \bigoplus_{\chi \in \Irr}{V^f_{\chi} \otimes X_{\chi}}$.\\
\end{itemize}

\noindent Moreover $\mathcal{F}$ sends a morphism $(\phi_{\chi},\psi_{\chi})$ to $(\Phi,\Psi)$ where
\begin{itemize}
\item $\Phi := \bigoplus_{\chi \in \Irr}{\phi_\chi\otimes \mathrm{id}_{X_{\chi}}}$
\item $\Psi := \bigoplus_{\chi \in \Irr}{\psi_\chi\otimes \mathrm{id}_{X_{\chi}}}$.\\
\end{itemize}
Now, let
\begin{equation*}
\mathcal{G}:
\begin{array}{ccc}
\RepGa & \to & \Repm\\
(M,M^f,\Delta,Z_1,Z_2)& \rightsquigarrow & (V, V^f, x, v^1, v^2) \\
\end{array}
\end{equation*}
where
\begin{itemize}
\item $V=\bigoplus_{\chi \in \Irr}{\Hom_{\Gamma}(X_{\chi},M)}$,
\item $V^f=\bigoplus_{\chi \in \Irr}{\Hom_{\Gamma}(X_{\chi},M^f)}$,
\item For all $h \in \overline{E}$, define $x_h$ such that the following diagram commutes
\begin{center}
$\begin{tikzcd}
\Hom_{\Gamma}(X_{h'},M) \ar[d, "-\otimes \mathrm{id}_{\Xstd}"'] \ar[r, "x_h"]& \Hom_{\Gamma}(X_{h''},M)\\
\Hom_{\Gamma}(\Xstd \otimes X_{h'}, \Xstd \otimes M) \ar[r,"{- \circ \tilde{y}_h^0}"']& \Hom_{\Gamma}(X_{h''},\Xstd \otimes M) \ar[u, "\Delta\circ -"']
\end{tikzcd}$,
\end{center}
\item $\forall \chi \in \Irr$
\begin{center}
$v^1_{\chi}:\begin{array}{ccc}
\Hom_{\Gamma}(X_{\chi},M^f) & \to & \Hom_{\Gamma}( X_{\chi},M)\\
f & \mapsto & \delta_{\chi}Z_1\circ f  \\
\end{array}$,
\end{center}
\item $\forall \chi \in \Irr$
\begin{center}
$v^2_{\chi}: \begin{array}{ccc}
\Hom_{\Gamma}(X_{\chi},M) & \to & \Hom_{\Gamma}\left(X_{\chi},M^f \right)\\
f & \mapsto & Z_2 \circ f  \\
\end{array}$.
\end{center}
\end{itemize}

\noindent Moreover $\mathcal{G}$ sends a morphism $(\Phi,\Psi)$ to $(\phi_{\chi},\psi_{\chi})_{\chi \in \Irr}$ where
\begin{itemize}
\item  $\phi_{\chi} := \Phi\circ -$
\item  $\psi_{\chi} := \Psi\circ -$.\\
\end{itemize}

\begin{prop}
\label{equiv cat}
The functors $\mathcal{F}$ and $\mathcal{G}$ define an equivalence of categories.\\
\end{prop}

\noindent The proof of Proposition \ref{equiv cat} is provided in Section \ref{appendix}.\\

\begin{rmq}
If $M,M^f$ are two $\Gamma$-modules, then it is clear that the functor $\mathcal{G}$ induces a morphism ${\RepQ_{M,M^f} \to \RepQ_{\mathcal{G}(M),\mathcal{G}(M^f)}}$. Moreover, if $V$ and $V^f$ are two $\Irr$-graded complex vector spaces, it is clear that $\mathcal{F}$ induces a morphism $\RepQ_{V,V^f} \to \RepQ_{\mathcal{F}(V),\mathcal{F}(V^f)}$.

Note that $\mathcal{F}$ only depends on $y^0$ and that $\mathcal{G}$ only depends on $\tilde{y}^0$ . When this dependence is not clear from context, we will respectively denote $\mathcal{F}$ and $\mathcal{G}$ by $\mathcal{F}_{y^0}$ and $\mathcal{G}_{y^0}$. The following proposition show that these choices account for the choice of orientation of the quiver underlying $\overline{Q_{\Gamma^f}}$.\\
\end{rmq}

\noindent We can now wonder what happens to the symplectic structures.\\

\begin{prop}
\label{link_symp}
For every orientation $\Omega$ of $G_{\Gamma}$ , there exists a family $(\lphoe_{h})_{h \in \overline{E}} \in {(\C^*)}^{|\overline{E}|}$ such that for all $(x,v^1,v^2),(\tilde{x},\tilde{v}^1,\tilde{v}^2) \in \RepQ_{V,V^f}$:
\begin{equation*}
\omega_\varepsilon\left((x,v^1,v^2),(\tilde{x},\tilde{v}^1,\tilde{v}^2)\right)=\\\omega\left(\mathcal{F}_{y}(x,v^1,v^2),\mathcal{F}_{y}(\tilde{x},\tilde{v}^1,\tilde{v}^2)\right)
\end{equation*}
and for all  $(\Delta,Z_1,Z_2),(\tilde{\Delta},\tilde{Z}_1,\tilde{Z}_2) \in \RepQ_{M,M^f}$:
\begin{equation*}
\omega_\varepsilon\left(\mathcal{G}_{\tilde{y}}(\Delta,Z_1,Z_2),\mathcal{G}_{\tilde{y}}(\tilde{\Delta},\tilde{Z}_1,\tilde{Z}_2)\right)=\\\omega\left((\Delta,Z_1,Z_2),(\tilde{\Delta},\tilde{Z}_1,\tilde{Z}_2)\right)
\end{equation*}
where $y_{h}:=\lphoe_{h}y^0_{h},\quad \tilde{y}_{h}:={\lphoe_{h}}^{-1}\tilde{y}^0_{h}$, for all $h \in \overline{E}$.
\end{prop}
\begin{proof}
Expliciting the right-hand side of the first equality gives
\begin{align*}
&\sum_{h \in \overline{E}}{\lphoe_{h} \lphoe_{\bar{h}} \Tr(x_{\bar{h}} \otimes [y_{\bar{h}}^0]_1 \circ \tilde{x}_{h} \otimes [y_{h}^0]_2 - x_{\bar{h}} \otimes [y_{\bar{h}}^0]_2 \circ \tilde{x}_{h} \otimes [y_{h}^0]_1)} + \Tr(\sum_{\chi \in \Irr}{v^1_{\chi}\tilde{v}^2_{\chi}-\tilde{v}^1_{\chi}v^2_{\chi}})\\
&= \sum_{h \in \overline{E}}{\lphoe_{h}\lphoe_{\bar{h}}\Tr([y_{\bar{h}}^0]_1[y_{h}^0]_2-[y_{\bar{h}}^0]_2[y_{h}^0]_1)\Tr(x_{\bar{h}}\tilde{x}_{h})} + \Tr(\sum_{\chi \in \Irr}{v^1_{\chi}\tilde{v}^2_{\chi}-\tilde{v}^1_{\chi}v^2_{\chi}})
\end{align*}

\noindent where we denote $[y_{h}^0]_k \in \Hom(X_{h'},X_{h''})$ the morphism $[y_{h}^0]_{e_k}$ which comes from the canonical identification of $T_h$ with $\Hom_{\Gamma}(\Xstd,\Hom(X_{h'},X_{h''}))$ and taking the morphism corresponding to $e_k \in \Xstd$.

Take $h \in \overline{E}$. The term $\Tr([y_{\bar{h}}^0]_1[y_{h}^0]_2-[y_{\bar{h}}^0]_2[y_{h}^0]_1)$ is nonzero thanks to the nondegenerescence of $\omega$ and Proposition \ref{equiv cat}. Moreover
\begin{equation*}
\Tr\left([y_{\bar{h}}^0]_1[y_{h}^0]_2-[y_{\bar{h}}^0]_2[y_{h}^0]_1\right)+ \Tr\left([y_{h}^0]_1[y_{\bar{h}}^0]_2-[y_{h}^0]_2[y_{\bar{h}}^0]_1\right)=0.
\end{equation*}
thanks to the well-known fact that $\Tr(AB)=\Tr(BA)$.

Now, we can choose $\lphoe_h$ and $\lphoe_{\bar{h}}$ in such a way that ${\lphoe_{h}\lphoe_{\bar{h}}\Tr([y_{\bar{h}}^0]_1[y_{h}^0]_2-[y_{\bar{h}}^0]_2[y_{h}^0]_1)=\varepsilon(\bar{h})}$. Note that the previous equation gives that the pair $(\lphoe_h,\lphoe_{\bar{h}})$ is unique up to a non-zero scalar.
To prove the second statement, we use Proposition \ref{equiv cat} and its proof to obtain that
\begin{center}
$\epsilon_M\circ \mathcal{F}(\mathcal{G}(\Delta))\circ \left(\mathrm{id}_{\Xstd}\otimes \epsilon_M^{-1}\right)=\Delta$,\\
$\epsilon_M\circ  \mathcal{F}(\mathcal{G}(Z_1)) \circ \epsilon_{M^f}^{-1}=Z_1$,\\
$\epsilon_{M^f}\circ \mathcal{F}(\mathcal{G}(Z_2)) \circ \epsilon_M^{-1}= Z_2$.
\end{center}
Use what has been proven to get
\begin{equation*}
\omega_\varepsilon\left(\mathcal{G}_{\tilde{y}}(\Delta,Z_1,Z_2),\mathcal{G}_{\tilde{y}}(\tilde{\Delta},\tilde{Z}_1,\tilde{Z}_2)\right)=\\\omega\left(\mathcal{F}_{y}\left(\mathcal{G}_{\tilde{y}}\left(\Delta,Z_1,Z_2\right)\right),\mathcal{F}_{y}\left(\mathcal{G}_{\tilde{y}}\left(\tilde{\Delta},\tilde{Z}_1,\tilde{Z}_2\right)\right)\right).
\end{equation*}
Using the invariance of the trace by conjugacy we get the desired equality.
\end{proof}

\noindent Fix an orientation $\Omega$ of the McKay multigraph $G_{\Gamma}$. We now have two families $(y_h)$, $(\tilde{y}_h)$ which will be used from now on.
The functors $\mathcal{F}$ and $\mathcal{G}$ intertwine the actions of $\Aut_{\Gamma}(M)$ on $\RepQ_{M,M^f}$ and of $\G(V)$ on $\RepQ_{V,V^f}$. Consider the morphism:\\
\begin{center}
${\tilde{p}_V:\begin{array}{ccc}
\bigoplus_{\chi \in \Irr}{\End(V_{\chi})} & \to & \End_{\Gamma}(\bigoplus_{\chi \in \Irr}{V_{\chi}\otimes X_{\chi}})\\
 (f_{\chi})_{\chi \in \Irr}& \mapsto & \left(\sum_{\chi \in \Irr}{v_{\chi}\otimes x_{\chi}} \mapsto \sum_{\chi \in \Irr}{f_{\chi}(v_{\chi})\otimes x_{\chi}}\right)\\
\end{array}}$
\end{center}
\vspace*{0.25cm}
 where $V$ is an $\Irr$-graded complex vector space.

Using Schur's Lemma it is clear that $\tilde{p}_V$ is an isomorphism and by construction it restricts to $p_V: G(V) \iso \Aut_{\Gamma}(\mathcal{F}(V))$ which is an isomorphism of algebraic groups.\\
If $M$ is a $\Gamma$-module, consider ${\tilde{p}_M: \End_{\Gamma}(M)  \iso  \bigoplus_{\chi \in \Irr}{\End(\Hom_{\Gamma}(X_{\chi},M))}}$ the conjugation by the isomorphism $\epsilon^{-1}_M$ (the isotypical decomposition of $M$). The morphism $\tilde{p}_M$ restricts to an isomorphism of algebraic groups $p_M:\Aut_{\Gamma}(M) \iso \G(\mathcal{G}(M))$.\\

\begin{prop}
\label{link_action}
Let $V$ and $V^f$ be two $\Irr$-graded complex vector spaces. If one lets the algebraic group $\G(V)$ act on $\RepQ_{\mathcal{F}(V),\mathcal{F}(V^f)}$ through $p_V$, then the map $\RepQ_{V,V^f} \to \RepQ_{\mathcal{F}(V),\mathcal{F}(V^f)}$ is $\mathrm{G}(V)$-equivariant.\\
Let $M$ and $M^f$ be two $\Gamma$-modules. If one lets the algebraic group $\Aut_{\Gamma}(M)$ act on $\RepQ_{\mathcal{G}(M), \mathcal{G}(M^f)}$ through $p_M$, then the map $\RepQ_{M,M^f} \to \RepQ_{\mathcal{G}(M), \mathcal{G}(M^f)}$ is $\Aut_{\Gamma}(M)$-equivariant.
\end{prop}
\begin{proof}
The result follows directly from the definitions of the actions on $\RepQ_{V,V^f}$ and $\RepQ_{M,M^f}$ and from the definition of the functors $\mathcal{F}$ and $\mathcal{G}$.
\end{proof}

\noindent We can also link the two momentum maps. The proof of the following proposition is provided in Section \ref{appendix}.\\
\begin{prop}
\label{link_momentum}
For any two $\Irr$-graded complex vector spaces $V$ and $V^f$, the following diagram commutes
\begin{center}
$\begin{tikzcd}[column  sep=4cm]
\RepQ_{V,V^f} \ar[r, "\mathcal{F}"] \ar[d, "\mu_{\varepsilon}"'] & \RepQ_{\mathcal{F}(V),\mathcal{F}(V^f)} \ar[d, "\mu"]\\
\bigoplus_{\chi \in \Irr}{\End(V_{\chi})} \ar[r,"\sim", "\sum_{\chi \in \Irr}{- \otimes \frac{1}{\delta_{\chi}}\mathrm{id}_{X_{\chi}}}"'] & \End_{\Gamma}(\mathcal{F}(V))
\end{tikzcd}$.
\end{center}
For any two $\Gamma$-modules $M$ and $M^f$, the following diagram commutes
\begin{center}
$\begin{tikzcd}[column sep=5cm]
\RepQ_{M,M^f} \ar[r, "\mathcal{G}"] \ar[d, "\mu"'] & \RepQ_{\mathcal{G}(M),\mathcal{G}(M^f)} \ar[d, "\mu_{\varepsilon}"]\\
\End_{\Gamma}(M)  & \bigoplus_{\chi \in \Irr}{\End\left(\mathcal{G}(M)_{\chi}\right)} \ar[l,"\sim"', "\epsilon_M \circ \left(\sum_{\chi \in \Irr}{- \otimes \frac{1}{\delta_{\chi}}\mathrm{id}_{X_{\chi}}}\right) \circ \epsilon_M^{-1}"] \end{tikzcd}$.
\end{center}
\end{prop}

\begin{rmq}
One might wonder why the map $\bigoplus_{\chi \in \Irr}{\End(V_{\chi})} \to \End_{\Gamma}(\mathcal{F}(V))$ is not $\tilde{p}_V$. This comes from the fact that the identifications between $\bigoplus_{\chi \in \Irr}{\End(V_{\chi})^*}$ and $\bigoplus_{\chi \in \Irr}{\End(V_{\chi})}$ and between $\End_{\Gamma}(\mathcal{F}(V))^*$ and $\End_{\Gamma}(\mathcal{F}(V))$ is not commuting with this map. Indeed, we need to define an isomorphism ${\bigoplus_{\chi \in \Irr}{\End(V_{\chi})} \to \End_{\Gamma}(\mathcal{F}(V))}$ that makes the following diagram commute
\begin{center}
\begin{tikzcd}
\End_{\Gamma}(\mathcal{F}(V)) \ar[r, "\mathrm{Tr}"] & \End_{\Gamma}(\mathcal{F}(V))^* \ar[d]\\
\bigoplus_{\chi \in \Irr}{\End(V_{\chi})} \ar[u] \ar[r, "\sum_{\chi \in \Irr}{\mathrm{Tr}}"']& \bigoplus_{\chi \in \Irr}{\End(V_{\chi})^*}
\end{tikzcd}.
\end{center}
\end{rmq}

\subsection{Nakajima's quiver varieties}
In this section, we will recall the general construction of Nakajima's quiver variety and use the setting of $\Gamma$-modules to introduce an analog. First, we introduce notation. Take $G$ an undirected multigraph and $d \in \Delta_G^+$. Denote for all $\nu \in I_G$, $V^{d}_{\nu}:= \C^{d_\nu}$. Let ${\G(d):=\G(V^d)}$ and for an $I_G$-graded complex vector space $V$ denote by $\dim(V)$ the dimension parameter $\sum_{\nu \in I_G}{\dim(V_{\nu})\nu} \in \Delta^+_G$.

If $H$ is a reductive algebraic group over $\C$, $\chi: H \to \C^*$ is a rational character and $X$ is a complex affine algebraic $H$-variety, then we will denote by $X^{\chi-ss}$ the set of all $\chi$-semistable points of $X$. We will also denote $X\sslash_{\chi} H$ the GIT quotient twisted by $\chi$.

For all $\theta \in \Theta$, there exists $N \in \mathbb{Z}_{\geq 1}$, such that $\mathrm{Im}(N\theta) \subset \mathbb{Z}$. If $\G(d)$ acts on an affine complex algebraic variety $X$, we will denote by $X^{\theta-ss}$ the $\chi_{N\theta}$-semistable points of $X$ and by $X \sslash _{\theta} \G(d)$ the GIT quotient $X \sslash_{\chi_{N\theta}} \G(d)$. Note that \cite[Corollary $9.15$]{Kir06} assures that the notion of $\theta$-semistability, does not depend on the choice of $N$.\\

\begin{deff}
 Define \textbf{Nakajima's quiver variety} of $\overline{Q_{G^f}}$ attached to an orientation $\Omega$ of $G$, to a dimension parameter $(d,d^f) \in \Delta^+$, to a stability parameter $\theta \in \Theta$ and to a deformation parameter $\lb \in \Lambda$ as follows
\begin{equation*}
\mathcal{M}^{\lb}_{\theta}\left(d, d^f\right) := {\mu_\varepsilon}^{-1}(\lb)\sslash_{\boldsymbol{\theta}} \G(d) \simeq {\mu_{\varepsilon}}^{-1}(\lb)^{\theta-ss}\sslash \G(d).
\end{equation*}

In the special case where $d^{f,i_0}:=\dim(V^{f,i_0})$ for some vertex $i_0 \in \Irr$, we will denote $\mathcal{M}^G_{\boldsymbol{\theta},\lb}\left(d,d^{f,i_0}\right)$ by $\mathcal{M}^G_{\boldsymbol{\theta},\lb}(d,i_0)$.\\
\end{deff}

\noindent Let us define an analog of Nakajima's quiver variety. If $M$ is a $\Gamma$-module and $\theta \in \Theta$, consider $\chi_{\theta}\circ p_M$ the rational character of $\Aut_{\Gamma}(M)$. If $X$ is an $\Aut_{\Gamma}(M)$-variety, we will also shorten $X^{\chi_{\theta}\circ p_M-ss}$ to $X^{\theta-ss}$.\\

\begin{rmq}
\label{rmq_def}
If $\lb \in \Lambda$, recall that we still denote by $\lb$ the induced element of $\bigoplus_{\chi \in \Irr}{\End(\mathcal{G}(M)_{\chi})}$. Let us also denote by $\lb$ the element of $\End_{\Gamma}(M)$ defined as
\begin{equation*}
\epsilon_M \circ \left(\sum_{\chi \in \Irr}{\lambda(\chi)\mathrm{id}_{{\mathcal{G}(M)}_{\chi}} \otimes \frac{1}{\delta_{\chi}}\mathrm{id}_{X_{\chi}}}\right) \circ \epsilon_M^{-1}.\\
\end{equation*}
\end{rmq}

\begin{deff}
Define \textbf{Nakajima's quiver variety} of $\DMQf$ attached to $\Gamma$-modules $M,M^f$, to a stability parameter $\theta \in \Theta$ and to a deformation parameter $\lambda \in \Lambda$ as follows
\begin{equation*}
\QVTdef\left(M,M^f\right) := \mu^{-1}(\lb)\sslash_{\theta} \Aut_{\Gamma}(M) \simeq \mu^{-1}(\lb)^{\theta-ss}\sslash \Aut_{\Gamma}(M).
\end{equation*}
Denote by $\QVTdef(M,\chi)$ the variety $\QVTdef(M,X_{\chi})$ for some $\chi \in \Irr$ and ${\QVTdef(M):=\QVTdef(M,\chi_0)}$.\\
\end{deff}

\begin{rmq}
\label{var_iso}
It is clear that if $M$ and $\tilde{M}$ are two $\Gamma$-modules and are isomorphic as $\Gamma$-modules, then for every $\Gamma$-module $M^f$, stability parameter $\theta \in \Theta$ and deformation parameter $\lb \in \Lambda$, we have ${\mathcal{M}^{\lb}_{\theta}\left(M,M^f\right) \simeq \mathcal{M}^{\lb}_{\theta}\left(\tilde{M},M^f\right)}$.\\
\end{rmq}

\noindent To be able to compare the two preceding quiver varieties, we need to show that $\mathcal{F}$ is compatible with semistability. Let us introduce a handy characterization of this notion.\\

\begin{deff}
Let $H$ be an algebraic group over $\C$, $\chi: H \to \C^*$ a rational character of $H$ and $\mu: \C^* \to H$ a one-parameter subgroup. Since the morphism $\chi\circ \mu: \C^* \to \C^*$ is of algebraic groups, there exists $k \in \mathbb{Z}$ such that $\forall t \in \C^*, \chi\left(\mu(t)\right)=t^k$. Define the pairing $\langle \chi, \mu \rangle$ to be equal to $k$.\\
\end{deff}

\begin{prop}
\label{mumford_crit}
A point $x \in X$ is $\chi$-semistable if and only if for all algebraic group morphism $\mu: \C^* \to H$, such that $\lim_{t\to 0}{\mu(t).x}$ exists, we have $\langle \chi, \mu \rangle \geq 0$.
\end{prop}
\begin{proof}
The proof can be found here \cite[Theorem 2.1]{Mum94}.
\end{proof}

\noindent We can use Proposition \ref{mumford_crit} to show that the functors $\mathcal{F}$ and $\mathcal{G}$ are compatible with semistability.\\
\begin{lemme}
\label{link_ss}
Take $V,V^f$ two $\Irr$-graded complex vector spaces and $M,M^f$ two $\Gamma$-modules.\\
The element ${(x,v^1,v^2) \in \RepQ_{V,V^f}}$ is $\chi_{\theta}$-semistable if and only if $\mathcal{F}(x,v^1,v^2)$ is ${\chi_{\theta}\circ p_{V}^{-1}}$-semistable. Moreover $(\Delta,Z_1,Z_2) \in \RepQ_{M,M^f}$ is $\chi_{\theta}\circ p_M$-semistable if and only if $\mathcal{G}(\Delta,Z_1,Z_2)$ is $\chi_{\theta}$-semistable.
\end{lemme}
\begin{proof}
Take  $\mu: \C^* \to \Aut_{\Gamma}\left(\mathcal{F}(V)\right)$ such that $\lim_{t\to 0}{\mu(t).\mathcal{F}(x,v^1,v^2)}$ exists. Using Proposition \ref{link_action}, we have for all $t \in \C^*$:
\begin{equation*}
\mu(t).\mathcal{F}(x,v^1,v^2)=\mathcal{F}\left(p_V^{-1}(\mu(t)).(x,v^1,v^2)\right).
\end{equation*}
Using the continuity of $\eta^{-1}\circ \mathcal{G}$ we have that $\lim_{t\to 0}{p_V^{-1}(\mu(t)).(x,v^1,v^2)}$ exists. Since $(x,v^1,v^2)$ is $\chi_{\theta}$-semistable by hypothesis, ${\langle \chi_{\theta}, p_V^{-1}\circ \mu \rangle = \langle \chi_{\theta} \circ p_V^{-1}, \mu \rangle \geq 0}$. Conversely we can use the equivariance and continuity of $\mathcal{F}$ to conclude.

The same arguments can be used if we start with $(\Delta,Z_1,Z_2) \in \RepQ_{M,M^f}$ and replaces $\mathcal{F}$ with $\mathcal{G}$ and $\eta^{-1}\circ \mathcal{G}$ with $\epsilon \circ \mathcal{F}$.
\end{proof}

\noindent We can now compare the defined varieties using $\mathcal{F}$.\\
\begin{thm}
\label{link_var}
For all parameters $((d,d^f),\theta,\lb)\in \Delta^+\times \Theta \times \Lambda$, $\mathcal{F}$ induces an isomorphism  $\tau:\QVTdef\left(d,d^f\right) \iso \mathcal{M}^{\lb}_{\theta}\left(\mathcal{F}\left(V^d\right),\mathcal{F}\left(V^{d^f}\right)\right)$.
\end{thm}
\begin{proof}
Consider the morphism $\hat{\tau}: \RepQ_{V,V^f} \to \RepQ_{\mathcal{F}(V),\mathcal{F}(V^f)}$ induced by $\mathcal{F}$. This map is linear and computing dimensions gives that ${\dim\left(\RepQ_{V,V^f}\right)=\dim\left(\RepQ_{\mathcal{F}(V),\mathcal{F}(V^f)}\right)}$. Let us show that $\hat{\tau}$ is injective. Take $(x,v^1,v^2) \in \mathrm{Ker}(\hat{\tau})$. By construction, we have that $v^1=0$ and $v^2=0$. Moreover, ${\mathcal{F}(x)=0}$ implies that ${\forall \chi \in \Irr, \sum_{h \in \overline{E}, h''=\chi}{y_h \otimes x_h}=0}$. Take $\chi \in \Irr$ and $h \in \overline{E}$ such that $h''=\chi$, then $y_h \neq 0$, and in particular $x_h=0$. The morphism $\hat{\tau}$ is then an isomorphism. Thanks to Proposition \ref{link_momentum} and Lemma \ref{link_ss}, we have ${\tilde{\tau}:{\mu_\varepsilon}^{-1}(\lb)^{\theta-ss} \iso \mu^{-1}(\lb)^{\theta-ss}}$ which is the restriction of the isomorphism $\hat{\tau}$ to ${{\mu_{\varepsilon}}^{-1}(\lb)^{\theta-ss}}$ . Finally, Proposition \ref{link_action} shows that $\tilde{\tau}$ induces the desired isomorphism.
\end{proof}

\noindent When not specified, the deformation parameter is taken to be $0 \in \Lambda$ and if the stability parameter is not specified it is also taken to be $0 \in \Theta$.

\subsection{Important results for McKay quiver varieties}
\label{imp_res}

In this subsection, we recall important results obtained for Nakajima's quiver varieties that will be used in the following sections. The next Proposition is a reformulation of \cite[Theorem 2.8]{Nak94} and of \cite[Section 1]{C-B01}.\\

\begin{prop}
\label{prop_irr}
Let $(d,d^f) \in \Delta^+$ be a dimension parameter. If ${(\theta,\lb)\in \Q \times \Lambda^{++}\setminus \{0,0\}}$, then $\QVTdef\left(d,d^f\right)$ is smooth and irreducible.
\end{prop}
\begin{proof}
If $\lb \neq 0$, \cite[Theorem $5.2.2$]{Ginz08} gives that $\QVTdef\left(d,d^f\right)$ is smooth and connected. Moreover, if $\theta \neq 0$, we have that  $\QVT\left(d,d^f\right)$ is smooth and connected so irreducible thanks to \cite[Example 10.36]{Kir06} and \cite[Theorem $10.35$, $10.37$]{Kir06}.
\end{proof}

\noindent One can easily compute the dimension of McKay quiver varieties when the dimension parameter is of a particularly nice form.\\

\begin{prop}
\label{dim_r_delt}
If $\theta \in \Theta^{++}$ and if $r$ is a positive integer, then  the symplectic variety $\mathcal{M}_{\theta}\left(r\delta\right)$ is smooth connected and has dimension $2r$.
\end{prop}
\begin{proof}
Let us apply \cite[Theorem $10.35$]{Kir06}. We have by definition that $A\delta=0$ so that the dimension of $\mathcal{M}_{\theta}(r\delta)$ is $2r\delta_{\chi_0}=2r$.
\end{proof}

\noindent In that case, semistability becomes also a much simpler condition.\\
\begin{lemme}
\label{lemme_ss}
Take $\theta$ in $\Theta^{++}$.\\
$(x,v^1,v^2) \in \RepQ_{d,d^f}$ is $\theta$-semistable if and only if for all $\chi \in \Irr$ and for all $S_{\chi} \subset V_{\chi}^d$, if $\left( \forall h \in \overline{E}, x_h(S_{h'}) \subset S_{h''} \text{ and } \forall \chi \in \Irr, \mathrm{Im}(v^1_{\chi}) \subset S_{\chi} \right)$, then $\bigoplus_{\chi \in \Irr}{S_{\chi}}= V^d$.\\
Moreover $(\Delta,Z_1,Z_2) \in \RepQ_{M,M^f}$ is $\theta$-semistable if and only if for all $\Gamma$-submodules $M'$ of $M$, if  $\Delta(\Xstd \otimes M') \subset M'$ and $\mathrm{Im}(Z_1) \subset M'$, then $M'=M$.
\end{lemme}
\begin{proof}
Let us proceed by contraposition. If we have that for all $\chi \in \Irr, S_{\chi} \subset V_{\chi}^d$ such that it is stable by $x$, that $\mathrm{Im}(v^1_{\chi}) \subset S_{\chi}$  and that $\bigoplus_{\chi \in \Irr}{S_{\chi}} \subsetneq V^d$. Let us construct $\mu: \C^* \to \G(d)$ such that $\lim_{t \to 0}{\lambda(t).(x,v^1,v^2)}$ exists and such that $\langle\chi_{\theta}, \mu\rangle < 0$.\\
For each $\xi \in \Irr$, consider $S_{\xi}^{\perp}$ a supplementary subspace of $S_{\xi}$ in $V^d_{\xi}$.
Define for each $\xi \in \Irr$ and each $\forall t \in \C^*$
\begin{equation*}
(\mu(t))_{\xi} :=
\begin{pmatrix}
   \mathrm{id}_{S_{\xi}}
  & \rvline & 0 \\
\hline
  0 & \rvline &
 t^{-1}\mathrm{id}_{S^{\perp}_{\xi}}
\end{pmatrix} \in \GL(V^d_{\xi})
\end{equation*}
Now, we use the fact that $\forall \xi \in \Irr, S_{\xi}$ is stable under $x$ and $\mathrm{Im}(v^1_{\xi}) \subset S_{\xi}$, to have the existence of  $\lim_{t\to 0}{\mu(t).(x,v^1,v^2)}$. Moreover, there exists $\chi \in \Irr$ such that $\mathrm{dim}(S_{\chi}^{\perp}) \geq 1$ so we have that $\langle\chi_{\theta},\mu \rangle < 0$ which contradicts Proposition \ref{mumford_crit} .\\
Conversely, take a one-parameter subgroup $\mu$ such that  $\lim_{t \to 0}{\mu(t).(x,v^1,v^2)}$ exists.
For all $\chi \in \Irr$, consider the eigenspace decomposition of $\mu$ acting on $V_{\chi}^d=\bigoplus_{k \in \mathbb{Z}}{V^d_{(\chi,k)}}$ meaning that $\forall t \in \C^*,\mu(t)\arrowvert_{V^d_{\chi,k}} = t^k\mathrm{id}_{V^d_{\chi,k}}$. For all $k \in \mathbb{Z}$ denote by $V^d_k:=\bigoplus_{\chi \in \Irr}{V^d_{\chi,k}}$ so that $V^d=\bigoplus_{k \in \mathbb{Z}}{V^d_k}$. Moreover, for $j \in \mathbb{Z}$, denote by $V^d_{\geq j}:=\bigoplus_{k \geq j}{\bigoplus_{\chi \in \Irr}{V^d_{\chi,k}}}$. Then we have $V^d_{\geq j}=\bigoplus_{\chi \in \Irr}{V^d_{\chi,\geq j}}$.\\
Let us prove that for all $k \in \mathbb{Z}$, $V^d_{\geq k}$ is stable by $x$. Take $h \in \overline{E}$, then $x_h:V^d_{h'} \to V^d_{h''}$ can be decomposed into the  direct sum $\bigoplus_{r,s \in \mathbb{Z}}{x_h^{(r,s)}}$ where ${x_h^{(r,s)}: V^d_{(h',r)} \to V^d_{(h'',s)}}$.
From there, for $t \in \C^*$, $\mu(t).x_h=\sum_{r,s \in \mathbb{Z}}{t^{s-r}x_h^{(r,s)}}$. The limit of $\mu(t).x_h$, when $t$ tends to $0$, exists if and only if $x_h^{(r,s)}=0$ for all pairs $(r,s)$ such that $s < r$. This gives that $V^d_{\geq r}$ is stable by $x_h$ for all $h \in \overline{E}$.

The same argument shows that the existence of the limit $\mu(t)_{\chi}.v^1_{\chi}$ implies that $\mathrm{Im}(v^1_{\chi}) \in V^d_{\chi, \geq 0}$ for all $\chi \in \Irr$. To resume, for each  $\chi \in \Irr$ we have an $x$-stable subspace $V^d_{\chi,\geq 0} \subset V^d_{\chi}$ and  $\mathrm{Im}(v^1_{\chi}) \in V^d_{\chi,\geq 0}$. Then by hypothesis, $V^d_{\geq 0} = V^d$. The conclusion follows, $\langle \chi_{\theta}, \mu \rangle \geq 0$.\\
For the second statement, take $(\Delta,Z_1,Z_2) \in \RepQ_{M,M^f}$ which is $\theta$-semistable. Take $M'$ a ${\Gamma\text{-submodule}}$ of $M$ stable by $\Delta$ and containing $\mathrm{Im}(Z_1)$. Thanks to Lemma \ref{link_ss}, we know that $\mathcal{G}(\Delta,Z_1,Z_2)$ is $\theta$-semistable. For each  $\chi \in \Irr$, consider ${S_{\chi}=\Hom_{\Gamma}(X_{\chi},M')}$ which is a subspace of $\Hom_{\Gamma}(X_{\chi},M)$. Since $\Delta(\Xstd \otimes M') \subset M'$, we have by construction of $\mathcal{G}$ that  $\left(\mathcal{G}(\Delta)_h \right)(S_{h'}) \subset S_{h''}$, for all $h \in \overline{E}$.
In addition, we also have that ${\forall \chi \in \Irr, \mathrm{Im}(\mathcal{G}(Z_1)_{\chi}) \subset S_{\chi}}$, since $\mathrm{Im}(Z_1) \subset M'$.

Let us use the first equivalence of this Lemma to deduce that ${\forall \chi \in \Irr, S_{\chi}=\Hom_{\Gamma}(X_{\chi},M)}$, which implies that $M'=M$. Conversely, using Lemma \ref{link_ss} it is enough to show that $\mathcal{G}(\Delta,Z_1,Z_2)$ is $\theta$-semistable. Denote $x:=\mathcal{G}(\Delta)$. Suppose that we have for each $\chi \in \Irr$ a subspace $S_{\chi}$ of $\Hom_{\Gamma}(X_{\chi},M)$ such that for all $h \in \overline{E}, x_h(S_{h'}) \subset S_{h''}$ and for all $\chi \in \Irr, \mathrm{Im}(\mathcal{G}(Z_1)_{\chi}) \subset S_{\chi}$.
Consider now $M'=\epsilon_{M}(\bigoplus_{\chi \in \Irr}{S_{\chi}\otimes X_{\chi}})$ which is a $\Gamma$-submodule of $M$. This submodule is stable by $\Delta$. Take $\chi \in \Irr$, and $h \in \overline{E}$ such that $h'=\chi$. By construction of $x$, we have
\begin{equation*}
x_h(f)=\Delta\circ (\mathrm{id}_{\Xstd} \otimes f) \circ \tilde{y}_h, \quad \forall f \in S_{\chi}.
\end{equation*}
Now
\begin{align*}
\sum_{h \in \overline{E},h'=\chi}{x_h(f)\circ y_h} &= \sum_{h \in \overline{E}}{\Delta\circ ( \mathrm{id}_{\Xstd} \otimes f) \circ \tilde{y}_h \circ y_h}\\
&= \Delta \circ (\mathrm{id}_{\Xstd} \otimes f) \circ \sum_{h \in \overline{E},h'=\chi}{\tilde{y}_h y_h}\\
&= \Delta \circ (\mathrm{id}_{\Xstd} \otimes f)
\end{align*}
The last equality follows from Lemma \ref{sum works}.\\
Then using the hypothesis on $x$, we have for all $(t,f,z) \in \Xstd \times S_{\chi}\times X_{\chi}$
\begin{equation*}
\Delta(t \otimes f(z))=\sum_{h \in \overline{E}, h'=\chi}{x_h(f)(y_h(t \otimes z)}) \in M'.
\end{equation*}
Finally, using the fact that $\forall \chi \in \Irr, \mathrm{Im}(\mathcal{G}(Z_1)_{\chi}) \in S_{\chi}$ it is clear that $\mathrm{Im}(Z_1) \subset M'$. By hypothesis, we have $M'=M$ and in particular that $\forall \chi \in \Irr, S_{\chi}=\Hom_{\Gamma}(X_{\chi},M)$ which shows that $\mathcal{G}(\Delta,Z_1,Z_2)$ is $\theta$-semistable.
\end{proof}

\noindent Let us finish this section by recalling an important isomorphism between McKay quiver varieties. The explicit realization described in Remark \ref{rmq_mckay}, gives geometric insights on the set of dimension parameters $\Delta$, the set of stability parameters $\Theta$ and the set of deformation parameters $\Lambda$. In fact, $(\Delta_{\Gamma})^{**}$ is equal to the root lattice $\mathcal{Q}$. This identifies the dimension parameter $\delta$ with the null root. Concerning $\Theta$, we need to introduce $P_{\mathbb{Q}}^{\vee}:=\{\beta \in \mathfrak{h} \mid \forall \chi \in \Irr, \langle \alpha_{\chi}, \beta \rangle \in \mathbb{Q}\}$. We can then identify $\Lambda$ with $\mathfrak{h}^*/\C\delta$ and $\Theta_{\Gamma}$ with $P_{\mathbb{Q}}^{\vee}/\Q\delta^{\vee}$ in the following way:\\
\begin{center}
$\kappa^{\vee}: \begin{array}{ccc}
P_{\Q}^{\vee}/\Q\delta^{\vee} & \to & \Theta\\
\beta^{\vee} & \mapsto & (\chi \mapsto \langle \alpha_{\chi}, \beta^{\vee} \rangle) \\
\end{array}$,\\
\end{center}
\begin{center}
$\kappa: \begin{array}{ccc}
\mathfrak{h}^*/\C\delta & \to & \Lambda\\
\beta & \mapsto & (\chi \mapsto \langle \beta,\tilde{\alpha}_{\chi} \rangle) \\
\end{array}$.
\end{center}
\vspace*{0.25cm}
Since $\langle \alpha_{\chi},\delta^{\vee} \rangle = 0 = \langle \delta, \tilde{\alpha}_{\chi} \rangle$ for all $\chi \in \Irr$, the morphisms $\kappa$ and $\kappa^{\vee}$ are well defined and it is easy to check that these are isomorphisms of $\mathbb{Z}$-modules.\\

\begin{deff}
\label{def_action}
Let us define an action of $W$ on $\Delta$, $\Theta$ and  $\Lambda$. Denote by $s_{\chi} \in W$, for $\chi \in \Irr$ the generators of this affine Weyl group. Take $d \in \Delta$, $\theta \in \Theta$ and $\lb \in \Lambda$

\begin{equation*}
(s_\chi.d)_{\xi} =
\begin{cases}
(\sum_{h\in \overline{E}, h'=\chi}{d_{h''}}) - d_{\chi} &\text{ if } \chi=\xi\neq \chi_0 \\
(\sum_{h\in \overline{E}, h'=\chi}{d_{h''}}) - d_{\chi} + 1 &\text{ if } \chi=\xi= \chi_0 \\
d_{\xi} &\text{ else},
\end{cases}
\end{equation*}

\begin{equation*}
(s_\chi.\theta)(\xi) =
\begin{cases}
\theta(\chi)+\theta(\xi) &\text{ if } \exists h\in \overline{E}, h'=\chi, h''=\xi\\
-\theta(\chi) &\text{ if } \chi=\xi\\
\theta(\xi) &\text{ else},
\end{cases}
\end{equation*}

\begin{equation*}
(s_\chi.\lambda)(\xi) =
\begin{cases}
\lambda(\chi)+\lambda(\xi) &\text{ if } \exists h\in \overline{E}, h'=\chi, h''=\xi\\
-\lambda(\chi) &\text{ if } \chi=\xi\\
\lambda(\xi) &\text{ else}.
\end{cases}
\end{equation*}

\vspace*{0.25cm}
\end{deff}
\begin{rmq}
\label{rmq_k_equi}
The group $W$ acts by reflections on $P_{\Q}^{\vee}$ and on $\mathfrak{h}^*$. Moreover $\delta^{\vee}$ and $\delta$ are stabilized by $W$. The actions defined on $\Theta$ and on $\Lambda$ turn the isomorphisms $\kappa^{\vee}$ and $\kappa$ into $W$-equivariant isomorphisms.
Moreover, the action on $\Delta$ corresponds to the one defined in \cite[Definition 2.3]{Nak03} in the special case of double, one vertex framed quivers and it is linked to the natural action by reflections on $\mathfrak{h}^*$ (denoted $*$) in the following way. Thanks to the remark at the end of \cite[Definition $2.3$]{Nak03}, we have
\begin{equation}
\label{link_eq}
\forall (\omega,\alpha) \in W \times \mathcal{Q},  \omega*(\Lambda_0-\alpha) = \Lambda_0 - \omega.\alpha
\end{equation}
where $\Lambda_0$ denotes the fundamental weight $\Lambda_{\alpha_0}$.\\
\end{rmq}

\noindent One important isomorphism between Nakajima quiver varieties, for dimension stability and deformation parameters that are linked by the actions of $W$ defined in Definition \ref{def_action}, was discovered by George Lusztig \cite[Corollary $3.6$]{Lusz00}, Andrea Maffei \cite[Proposition 40]{Maff00} and Hiraku Nakajima \cite[Theorem $8.1$]{Nak03}.\\

\begin{prop}
\label{maff}
Let $(d,\theta,\lb) \in \Delta^+  \times \Theta^{++} \times \Lambda$, and $\omega \in W$, such that $\omega.d \in \Delta^+$. We then have an isomorphism
\begin{equation*}
\mathrm{Maff}: \mathcal{M}^{\lb}_{\theta}(d) \iso \mathcal{M}^{\omega.\lb}_{\omega.\theta}(\omega.d)
\end{equation*}
of algebraic varieties.\\
\end{prop}

\noindent Finally, Ivan Losev \cite[Lemma $6.4.2$]{L12} has proved the following result.\\
\begin{prop}
\label{losev}
Let $(d,\theta) \in \Delta^+ \times \Theta^{++}$, and $\omega \in W$, such that $\omega.d=d$. The morphism
\begin{equation*}
\mathrm{Maff}: \mathcal{M}_{\theta}(d) \iso \mathcal{M}_{\theta}^{\Gamma}(d)
\end{equation*}
is an isomorphism of algebraic varieties over $\mathcal{M}(d):=\mathcal{M}_0(d)$.
\end{prop}

\section{Decomposition of $\Gamma$-fixed point loci}
\label{chap_decomp}

The Jordan quiver is the quiver with one vertex and one arrow. The name of this quiver comes from the fact that, over an algebraically closed field, the classification of the representations of this quiver is given by the Jordan normal form of a matrix. We will define an action of the group $\Gamma$ on this quiver variety and describe the irreducible components of the $\Gamma$-fixed point locus of the Jordan quiver variety for nonzero stability or nonzero deformation parameter. The irreducible components will be identified with McKay quiver varieties. After introducing the Jordan quiver variety in the first subsection, the second subsection is dedicated to the construction of a morphism from the $\Gamma$-fixed point locus of the Jordan quiver variety to the representation space $\RepQ_{M}$, for a $\Gamma$-module $M$ built out of the $\Gamma$-fixed point. In the third subsection, we will build a morphism from a McKay quiver variety to the $\Gamma$-fixed point locus of the Jordan quiver variety. Finally, the last subsection binds these constructions together to prove the main theorem of this section.
The setting developed in section \ref{chap_quiver}, will be of great use here.

\subsection{McKay and Jordan quiver}

From now on, we will mainly be interested in two types of double framed quivers. The first quiver is $\DMQf$, the double framed McKay quiver attached to a fixed finite subgroup $\Gamma$ of $\SL_2(\C)$. To be more specific, the double quiver of $Q_{\Gamma}$ framed at the vertex $\chi_0$ will be the main player. The second quiver is the Jordan quiver that will be denoted $Q_{\bullet}$
\begin{figure}[H]
  \centering
  \begin{tikzpicture}
    \draw[fill = black] (-0.7 ,0) circle (0.08);
    \draw ([shift=(160 : 0.7)]0, 0) arc (160 : 0: 0.7) node [anchor = west] {$\alpha$};
    \draw[->] (0.7, 0)  arc (0 : -160 :  0.7);
  \end{tikzpicture}
  \label{jordan quiver}
\end{figure}

\noindent Let us denote $G_{\bullet}$ the underlying undirected graph. Consider $\overline{Q_{\bullet^f}}$ the double framed quiver of $G_{\bullet}$
\begin{figure}[H]
  \centering
  \begin{tikzpicture}
    \draw[fill = black] (-0.7 ,0) circle (0.08);
    \draw[fill = black] (-2.5 ,0) circle (0.08) node [anchor = east] {$\infty$};
    \draw ([shift=(160 : 0.7)]0, 0) arc (160 : 0: 0.7) node [anchor = west] {$\alpha$};
    \draw ([shift=(-150 : 0.5)]0, 0) arc (-150 : 0: 0.5) node [anchor = east] {$\beta$};
    \draw (-0.9, 0.2) arc (0 : 90: 0.6 and 0.3 ) node [anchor = south] {$v^2$};
    \draw (-2.4, -0.2) arc (-180 : -90: 0.6 and 0.3 ) node [anchor = north] {$v^1$};
    \draw[->] ([shift=(90:0.5)]-1.5, 0)  arc (90 : 170 :  0.9 and 0.3);
    \draw[->] ([shift=(-90:0.5)]-1.8, 0)  arc (-90 : 0 :  0.9 and 0.3);
    \draw ([shift=(-150 : 0.5)]0, 0) arc (-150 : 0: 0.5) node [anchor = east] {$\beta$};
    \draw[->] (0.5, 0)  arc (0 : 150 :  0.5);
    \draw[->] (0.7, 0)  arc (0 : -160 :  0.7);
  \end{tikzpicture}
  \label{jordan quiver double}
\end{figure}

\noindent In this case, the momentum map for the action of $\GL_n(\C)$ is given by:
\begin{center}
$\mu_{\bullet}:\begin{array}{ccc}
M_n(\C)^2 \oplus \C^n \oplus \Hom(\C^n,\C) & \to & M_n(\C)\\
(\alpha,\beta,v^1,v^2) & \mapsto & [\alpha,\beta] + v^1v^2
\end{array}$.\\
\end{center}
\vspace*{0.25cm}

\noindent Fix an integer $n \geq 1$. Let us denote $\Mbullet^{\hspace{0.3em}\lb}_{\hspace{0.3em}\theta}(n)$ the quiver variety over the quiver $\overline{Q_{\bullet^f}}$ with stability parameter $\theta \in \Q$ and deformation parameter $\lb \in \C$.\\

\begin{rmq}
\label{free action}
\begin{enumerate}[label=(\roman*)]
\item  The group $\GL(\C^n)$ acts freely on $\mu_{\bullet}^{-1}(0)^{1-ss}$ and on $\mu_{\bullet}^{-1}(1)$.
\item  If $\theta$ is nonzero, we have $\Mbullet_{\hspace{0.3em}\theta}(n) \simeq \QVJn$, thanks to \cite[Lemma $10.29$ \& Theorem $11.5$]{Kir06}. In addition, if $\lb$ is nonzero by rescaling we also have
\begin{equation*}
\Mbullet^{\hspace{0.3em}\lb}(n) \simeq \QVJndefu.
\end{equation*}
Furthermore, thanks to $(i)$ and the definition of semistability, it is clear that for each $\theta \in \Q^*$,  $\mu_{\bullet}^{-1}(1)={\mu_{\bullet}^{-1}(1)}^{\theta-ss}$.
We know \cite[Theorem $11.5$]{Kir06} that  ${\Mbullet\hspace{0.3em}(n):=\Mbullet^{\hspace{0.3em}0}_{\hspace{0.3em}0}(n) \simeq{(\C^2)}^n / \Sk_n}$. In the end, for each $(\theta, \lb) \in \mathbb{Q} \times \C$ we have\\
\begin{center}
$\QVJnb \simeq
\begin{cases}
\Mbullet^{\hspace{0.3em}1}(n) &\text{ if } \lb \neq 0,\\
\Mbullet_{\hspace{0.3em}1}(n) &\text{ if } \lb= 0 \text{ and } \theta \neq 0,\\
\Ctwo^n / \Sk_n &\text{ if } (\theta, \lb) =(0,0).\\
\end{cases} $
\end{center}
\end{enumerate}
\end{rmq}

\noindent Fix a couple $(\theta,\lb) \in \Q \times \C \setminus \{(0,0)\}$. As a result of Remark \ref{free action} and \cite[Lemma $10.29$]{Kir06}, if $\theta \neq 0$ we do not lose generality by assuming that $\theta >0$.
\begin{deff}
Let us define a $\GL_2(\C)$-action on $\mu_{\bullet}^{-1}(\lb)^{\theta-ss}$.
For ${g=\begin{pmatrix}a & b \\c &d\end{pmatrix}\in \GL_2(\C)}$ and $(\alpha,\beta,v^1,v^2) \in \mu_{\bullet}^{-1}(\lb)^{\theta-ss}$:
\begin{equation*}
g.(\alpha,\beta,v^1,v^2) := (a\alpha+ b\beta, c\alpha+d\beta, v^1,v^2).
\end{equation*}
\end{deff}
\begin{rmq}
It is easy to check that this action commutes with the $\G(n)=\GL(\C^n)$ action. The $\GL_2(\C)$-action thus descends to $\QVJnb$.
\end{rmq}

\subsection{Deconstruction}
\label{fixed to rep}

\noindent In this section, let us start with $\overline{(\alpha,\beta,v^1,v^2)} \in {\QVJnb}^{\Gamma}$.
Let $\tilde{A}_{\alpha,\beta}$ be $e_1 \otimes \alpha + e_2 \otimes \beta$  an element of $\Xstd \otimes \mathrm{End}(\C^n)$ and denote by $A_{\alpha,\beta}$ the image of $\tilde{A}_{\alpha,\beta}$ though this chain of canonical isomorphisms
\begin{equation*}
\Xstd \otimes \mathrm{End}(\C^n) \iso \Hom(\C,\Xstd) \otimes \Hom(\C^n,\C^n) \iso \Hom(\C^n, \Xstd \otimes \C^n)
\end{equation*}

\noindent When the couple $(\alpha,\beta)$ is defined by the context, we just write $A$ instead of $A_{\alpha,\beta}$.
For each $\gamma \in \Gamma$, there exists a $g_{\gamma} \in \GL(\C^n)$ such that
\begin{equation}
\label{equ 1}
\gamma.(\alpha, \beta, v^1,v^2) = g_{\gamma}.(\alpha, \beta,v^1,v^2).
\end{equation}
Note that thanks to Remark \ref{free action}, the element $g_{\gamma}$ is unique.
Consider now the following group morphism

\begin{equation*}
\sigma:
\begin{array}{ccc}
\Gamma & \to & \GL(\C^n)\\
\gamma & \mapsto & g_{\gamma}^{-1} \\
\end{array}.
\end{equation*}
This morphism equips $\C^n$ with a structure of $\Gamma$-module. Denote this $\Gamma$-module by $M^{\sigma}$.
We can reformulate $(\ref{equ 1})$ as follows
\begin{equation*}
\label{fix_eq}
\gamma.(\alpha, \beta, v^1,v^2) = \sigma(\gamma^{-1}).(\alpha, \beta,v^1,v^2), \quad \forall \gamma \in \Gamma.
\end{equation*}

\begin{lemme}
The morphism $A$ is $\Gamma$-equivariant.
\end{lemme}
\begin{proof}
Take $x \in M^{\sigma}$ and $\gamma=\begin{pmatrix}a & b \\c &d\end{pmatrix} \in \Gamma$, using equation (\ref{equ 1}), we have the following equalities
\begin{center}
$\begin{cases}
\alpha(g_{\gamma}^{-1}(x))  = g_{\gamma}^{-1}(a\alpha(x)+b\beta(x))\\
\beta(g_{\gamma}^{-1}(x)) =   g_{\gamma}^{-1}(c\alpha(x)+d\beta(x)),
\end{cases}$
\end{center}
which then gives
\begin{center}
$\begin{cases}
e_1 \otimes \alpha(g_{\gamma}^{-1}(x))  = e_1 \otimes g_{\gamma}^{-1}(a\alpha(x)+b\beta(x))\\
e_2 \otimes \beta(g_{\gamma}^{-1}(x)) = e_2 \otimes  g_{\gamma}^{-1}(c\alpha(x)+d\beta(x)).
\end{cases}$
\end{center}
Summing these two equations provides exactly that $A$ is $\Gamma$-equivariant.
\end{proof}

\noindent Consider the morphism
\begin{equation*}
\mathrm{det}: \begin{array}{ccc}
\Xstd \otimes \Xstd & \to_{\Gamma} & X_{\chi_0}\\
\begin{pmatrix} r_1 \\ r_2 \end{pmatrix} \otimes \begin{pmatrix} s_1 \\ s_2 \end{pmatrix} & \mapsto & r_1s_2-r_2s_1 \\
\end{array}.
\end{equation*}

\noindent We now have everything. Let us define $\Delta^A: \Xstd \otimes M^{\sigma} \to M^{\sigma}$ as the composition of the two following maps
\begin{center}
\begin{tikzcd}
\Xstd \otimes M^{\sigma} \ar[r, "\mathrm{Id} \otimes A"] & \Xstd \otimes \Xstd \otimes M^{\sigma} \ar[r, "\mathrm{det} \otimes \mathrm{Id}"] & M^{\sigma}
\end{tikzcd}.
\end{center}
\vspace*{0.25cm}

\begin{lemme}
The morphism $\Delta^A$ is $\Gamma$-equivariant.
\end{lemme}
\begin{proof}
This follows from the equivariance of $A$ and of $\mathrm{det}$.
\end{proof}

\noindent Finally, $v^1$ defines an element $Z_1^v$ of $\Hom_{\Gamma}(X_{\chi_0},M^{\sigma})$ since  $\forall \gamma \in \Gamma, v^1=\sigma(\gamma^{-1})v^1$. In the same way,  $v^2$ defines an element $Z_2^v$ of $\Hom_{\Gamma}(M^{\sigma},X_{\chi_0})$. Bringing everything together gives the following Proposition.\\
\begin{prop}
\label{constr rep}
For each $\overline{(\alpha,\beta,v^1,v^2)} \in {\QVJnb}^{\Gamma}$, $(\Delta^{A},Z_1^v,Z_2^v)$ is an element of $\RepQ_{M^{\sigma}}$.
\end{prop}

\subsection{Reconstruction}

Let us turn it the other way around. Take a $\Gamma$-module $M$ of dimension $n$.
Consider the morphism
\begin{center}
$\mathrm{ted}: \begin{array}{ccc}
X_{\chi_0}& \to_{\Gamma} & \Xstd \otimes \Xstd\\
1 & \mapsto & e_2\otimes e_1 - e_1 \otimes e_2  \\
\end{array}$.
\end{center}
\vspace*{0.25cm}
Take $(\Delta,Z_1,Z_2) \in \RepQ_{M}$ and let $A_{\Delta}$ be the composition
\begin{center}
\begin{tikzcd}
M \ar[r, "\mathrm{ted}\otimes \mathrm{Id}"] & \Xstd \otimes \Xstd \otimes M \ar[r, "\mathrm{Id} \otimes \Delta"] & \Xstd \otimes M.
\end{tikzcd}
\end{center}
\vspace*{0.25cm}
\begin{lemme}
\label{lemme_fix}
The morphism $A_{\Delta}$ is $\Gamma$-equivariant.
\end{lemme}
\begin{proof}
Since $\mathrm{ted}$ and $\Delta$ are $\Gamma$-equivariant, $A_{\Delta}$ also is $\Gamma$-equivariant.
\end{proof}

\noindent Define $(\alpha_\Delta,\beta_\Delta) \in \mathrm{End}(M)^2$ such that for all $m \in M$:
\begin{equation*}
A_{\Delta}(m)=e_1 \otimes \alpha_{\Delta}(m) + e_2 \otimes \beta_{\Delta}(m).
\end{equation*}
Let us denote $\RepQ^{\bullet}_M:=\mathrm{End}(M)^2 \oplus M \oplus \Hom(M,\C)$. We can now consider the following linear map
\begin{center}
$\tilde{\iota}_M: \begin{array}{ccc}
\RepQ_M & \to & \RepQ^{\bullet}_{M}\\
(\Delta,Z_1,Z_2) & \mapsto & (\alpha_{\Delta} ,\beta_{\Delta},Z_1 ,Z_2)  \\
\end{array}$.
\end{center}
\vspace*{0.25cm}

\begin{lemme}
\label{iota_equiv}
The map $\tilde{\iota}_M$ is $\Aut_{\Gamma}(M)$-equivariant.
\end{lemme}
\begin{proof}
Take $g \in \Aut_{\Gamma}(M)$ and $(\Delta,Z_1,Z_2) \in \RepQ_M$. By definition of the action of $\Aut_{\Gamma}(M)$ on $\RepQ_M$, we have that $(g.\Delta)_{e_2}=g\Delta_{e_2}g^{-1}$ and $(g.\Delta)_{e_1}=g\Delta_{e_1}g^{-1}$. Since $\alpha_{\Delta}=-\Delta_{e_2}$ and $\beta_{\Delta}=\Delta_{e_1}$, it is clear that, by definition of the $\GL(M)$-action on $\RepQ^{\bullet}_M$, the map $\tilde{\iota}_M$ is $\Aut_{\Gamma}(M)$-equivariant.
\end{proof}

\noindent Recall that if $\lb \in \Lambda$, we abuse the notation and denote by $\lb$ the associated element of $\Hom_{\Gamma}(M)$ (cf. Remark \ref{rmq_def}).\\

\begin{prop}
\label{reconstr works}
The map $\tilde{\iota}$ induces a closed embedding $\iota_M: \QV_{\theta}^{\lb\delta}(M) \to \QVJnb^{\Gamma}$.
\end{prop}
\begin{proof}
First let us explain why $\tilde{\iota}_M(\tQVTMdef) \subset \mu_{\bullet}^{-1}(\lb)$.\\
Take $(\Delta,Z_1,Z_2) \in \RepQ_M$. By construction $\alpha_{\Delta}$ and $\beta_{\Delta}$ are respectively the morphisms $-\Delta_{e_2}$ and $\Delta_{e_1}$.
We then have
\begin{equation*}
\alpha_{\Delta}\beta_{\Delta} -\beta_{\Delta}\alpha_{\Delta} + Z_1Z_2 = -\Delta_{e_2}\Delta_{e_1}+\Delta_{e_1}\Delta_{e_2} + Z_1Z_2.
\end{equation*}
This computation shows that if $(\Delta,Z_1,Z_2) \in \mu^{-1}(\lb\delta)$, then $(\alpha_{\Delta},\beta_{\Delta},Z_1,Z_2)\in \mu_{\bullet}^{-1}(\lb)$. Let $(\Delta, Z_1,Z_2)$ be an element of $\mu^{-1}(\lb\delta)^{\theta-ss}$. If $\lb\neq 0$ then Remark \ref{free action} gives directly that $\tilde{\iota}_M(\Delta,Z_1,Z_2)$ is $\theta$-semistable. Take $(\Delta,Z_1,Z_2) \in \mu^{-1}(0)^{\theta-ss}$. If $\theta=0$ there is nothing to show. If $\theta \neq 0$, we assume that $\theta >0$. The Lemma \ref{lemme_ss} implies that $Z_1 \neq 0$ and using \cite[Lemma $11.6$]{Kir06} we have that $\alpha_{\Delta}$ and $\beta_{\Delta}$ commute. To show that $(\alpha_{\Delta},\beta_{\Delta},Z_1,Z_2)$ is $\theta$-semistable, using \cite[Example $10.36$]{Kir06}, it is enough to show that $M=\C[\alpha_{\Delta},\beta_{\Delta}]\mathrm{Im}(Z_1)$.
Let $S=\C[\alpha,\beta]\mathrm{Im}(Z_1)$, the construction of $\alpha_{\Delta}$ and $\beta_{\Delta}$ and the fact that $\mathrm{Im}(Z_1)$ is a $\Gamma$-submodule of $M$, makes it clear that $S$ is a $\Gamma$-submodule of $M$. Moreover, if $x=x_1e_1+x_2e_2 \in X_{\mathrm{std}}$ and $s \in S$, then $\Delta(x\otimes s)=x_1\beta_{\Delta}(s)-x_2\alpha_{\Delta}(s)\in S$. So, using Lemma \ref{lemme_ss}, we have that $S=M$.
The morphism  $\tilde{\iota}_M$ is $\Aut_{\Gamma}(M)$-equivariant (Lemma \ref{iota_equiv}) . It then induces ${\iota_M: \QVTMb \to \QVJnb}$.
There remains to show that $\iota_{M}(\QVTMb) \subset \QVJnb^{\Gamma}$. Indeed, if $(\Delta,Z_1,Z_2) \in \RepQ_M$, ${\gamma=\begin{pmatrix} a & b \\ c & d \end{pmatrix} \in \Gamma}$ and $m \in M$, then Lemma \ref{lemme_fix} implies that
\begin{equation*}
\gamma.(e_1\otimes \alpha_{\Delta}(m)+e_2\otimes \beta_{\Delta}(m))=e_1\otimes \alpha_{\Delta}(\gamma.m) + e_2 \otimes \beta_{\Delta}(\gamma.m).
\end{equation*}
From there, we have
\begin{center}
$\begin{cases}
\gamma.(a\alpha_{\Delta}(m)+b\beta_{\Delta}(m))=\alpha_{\Delta}(\gamma.m)\\
\gamma.(c\alpha_{\Delta}(m)+d\beta_{\Delta}(m))=\beta_{\Delta}(\gamma.m)
\end{cases}
\Rightarrow
\begin{cases}
a\alpha_{\Delta}(m)+b\beta_{\Delta}(m)=\gamma^{-1}.\alpha_{\Delta}(\gamma.m)\\
c\alpha_{\Delta}(m)+d\beta_{\Delta}(m)=\gamma^{-1}.\beta_{\Delta}(\gamma.m).
\end{cases}
$
\end{center}
\end{proof}

\subsection{Synthesis}

Let us now connect the last two sections. Let $L_{\oando}$ be the following algebraic variety
\begin{equation*}
 \left\{(\alpha,\beta,v^1,v^2,\sigma) \in \mu_{\bullet}^{-1}(\lb)^{\theta-ss} \times \RepGn \text{ } |\text{ } \forall \gamma \in \Gamma, {\gamma}.(\alpha,\beta,v^1,v^2) = \sigma(\gamma^{-1}).(\alpha,\beta,v^1,v^2)\right\}
\end{equation*}

\noindent and define\\
\begin{itemize}
\item$\tilde{p}_{\oando}: L_{\oando} \to \mu_{\bullet}^{-1}(\lb)^{\theta-ss}$,
\item $p_{\oando}: L_{\oando} \to \QVJnb = \pi \circ \tilde{p}_{\oando}$,
\item $q_{\oando}: L_{\oando} \to \RepGn$.
\end{itemize}
\vspace*{0.25cm}

To clarify the situation, we give here a diagram clarifying the previously introduced morphisms:
\begin{center}
\begin{tikzcd}
& &L_{\oando} \ar[dl,"\tilde{p}_{\oando}"'] \ar[dr,"q_{\oando}"] \ar[ddll, bend right=40, "p_{\oando}"']&\\
&  \mu_{\bullet}^{-1}(\lb)^{\theta-ss} \ar[dl,"\pi"']&&\RepGn\\
\QVJnb & &&
\end{tikzcd}.
\end{center}
\vspace*{0.25cm}
Let the group $\GL_n(\C)$ act by conjugacy on $\RepGn$ and diagonally on $L_{\oando}$.
The maps $\tilde{p}_{\oando}$ and $q_{\oando}$ are $\GL_n(\C)$-equivariant maps.
Moreover, note that ${p_{\oando}(L_{\oando})=\QVJnb^{\Gamma}}$.

\begin{lemme}
\label{inj}
If $(\alpha,\beta,v^1,v^2,\sigma)$ and $(\alpha,\beta,v^1,v^2, \sigma')$ are both in $L_{\oando}$, then $\sigma=\sigma'$.
\end{lemme}
\begin{proof}
From the assumption, we have $\forall \gamma \in \Gamma$ $\sigma(\gamma).(\alpha,\beta,v^1,v^2)=\sigma'(\gamma).(\alpha,\beta,v^1,v^2)$. We can then use Remark \ref{free action} to deduce that $\sigma=\sigma'$.
\end{proof}

\noindent If $\sigma \in \RepGn$ recall that we denote by $M^{\sigma}:=\C^n$ the $\Gamma$-module induced by $\sigma$.\\

\begin{deff}
Let $\sigma \in \RepGn$ be such that $q_{\oando}^{-1}(\sigma) \neq \emptyset$ and define
\begin{center}
$\tilde{\kappa}_{\sigma}: \begin{array}{ccc}
q_{\oando}^{-1}(\sigma) & \to & \RepQ_{M^{\sigma}}\\
(\alpha,\beta,v^1,v^2,\sigma) & \mapsto & (\Delta^{A},Z_1^v,Z_2^v)  \\
\end{array}$
\end{center}
where $(\Delta^{A},Z_1^v,Z_2^v)$ is as in Proposition \ref{constr rep}.\\
\end{deff}

\begin{prop}
\label{rep works}
If $\sigma \in \RepGn$ such that $q_{\oando}^{-1}(\sigma) \neq \emptyset$, then $\tilde{\kappa}_{\sigma}\left(q_{\oando}^{-1}(\sigma)\right) = \mu^{-1}(\lb\delta)^{\theta-ss}$.
\end{prop}
\begin{proof}
Take $(\alpha,\beta,v^1,v^2,\sigma) \in q_{\oando}^{-1}(\sigma)$, by construction $\Delta^A_{e_1}=\beta$ and $\Delta^A_{e_2}=-\alpha$. We then have
\begin{equation*}
\Delta^A_{e_1}\Delta^A_{e_2}-\Delta^A_{e_2}\Delta^A_{e_1} + Z_1^vZ_2^v = -\beta\alpha + \alpha\beta + v^1v^2.
\end{equation*}
This proves that if $(\alpha,\beta,v^1,v^2,\sigma)\in q^{-1}_{\oando}(\sigma)$ then  $\tilde{\kappa}_{\sigma}(\alpha,\beta,v^1,v^2,\sigma) \in \mu^{-1}(\lb\delta)$.

If $\theta=0$ it is clear that $\tilde{\kappa}_{\sigma}^{\Gamma}(\alpha,\beta,v^1,v^2,\sigma) \in  \mu^{-1}(\lb\delta)^{\theta-ss}$. If now $\theta \neq 0$, thanks to Remark \ref{free action}, we do not lose generality by assuming that $\lb=0$. Moreover, recall that in the case $\theta \neq 0$, we can reduce to the case where $\theta >0$ so that  Lemma \ref{lemme_ss} can be used to show that $(\Delta^A,Z_1^v,Z_2^v)$ is $\theta$-semistable. Take $M'$ a $\Gamma$-submodule of $M^{\sigma}$ such that $\Delta^{A}(\Xstd \otimes M') \subset M'$ and $\mathrm{Im}(Z_1^v) \subset M'$. Let us show that $M^{\sigma} \subset M'$. Since the element $(\alpha,\beta,v^1,v^2)$ is in  $\mu_{\bullet}^{-1}(0)^{\theta-ss}$, we have that $[\alpha,\beta]=0$ and $\C[\alpha,\beta]v^1=M^{\sigma}$. Take $m \in M^{\sigma}$, we then have a polynomial $P \in \C[x,y]$ such that $P(\alpha,\beta)v^1=m$. Since $\alpha = -\Delta^A_{e_2}$ and $\beta= \Delta^A_{e_1}$, $m=P(-\Delta_{e_2}^A,\Delta_{e_1}^A)v$. By hypothesis $v^1 \in M'$ and using the stability of $M'$ by $\Delta$ we can conclude that $m \in M'$.

To finish, let us prove the other inclusion i.e. if $(\Delta,Z_1,Z_2) \in \mu^{-1}(\lb\delta)^{\theta-ss}$ we need to show that $(-\Delta_{e_2},\Delta_{e_1},Z_1,Z_2,\sigma) \in L_{\oando}$. Take $\gamma \in \Gamma$, then a quick computation gives
\begin{align*}
\gamma.(-\Delta_{e_2},\Delta_{e_1},Z_1,Z_2)
&= (-\Delta_{\gamma^{-1}e_2}, \Delta_{\gamma^{-1}e_1},Z_1,Z_2)\\
&= \left(-\sigma(\gamma)^{-1}\Delta_{e_2}\sigma(\gamma), \sigma(\gamma)^{-1}\Delta_{e_1}\sigma(\gamma),\sigma(\gamma)^{-1}Z_1,Z_2\sigma(\gamma)\right).
\end{align*}
The last equality comes from the $\Gamma$-equivariance of $\Delta$, $Z_1$ and $Z_2$.
\end{proof}

\noindent Consider now $\kappa_{\sigma}: q_{\oando}^{-1}(\sigma) \to \mathcal{M}^{\lb\delta}_{\theta}(M^{\sigma})$ which is onto thanks to Proposition \ref{rep works}.
\begin{prop}
 \label{prop_commutes}
For each $\sigma \in \mathrm{Rep}_{n}$ such that $q_{\oando}^{-1}(\sigma) \neq \emptyset$, the following diagram commutes
\begin{center}
\begin{tikzcd}[column sep=huge]
 q_{\oando}^{-1}(\sigma) \ar[d, "p_{\oando}"] \ar[r, "\kappa_{\sigma}"] & \mathcal{M}_{\theta}^{\lb\delta}(M^{\sigma}) \ar[dl, "\iota_{M^{\sigma}}"]\\
 \QVJnb^{\Gamma}
 \end{tikzcd}.
 \end{center}
 \end{prop}
 \begin{proof}
Take $(\alpha, \beta,v^1,v^2, \sigma) \in q_{\oando}^{-1}(\sigma)$. By construction one has
\begin{equation*}
\tilde{\iota}_{M^{\sigma}}\left(\tilde{\kappa}_{\sigma}(\alpha,\beta,v^1,v^2,\sigma)\right)=(\alpha,\beta,v^1,v^2)
\end{equation*}
 which shows that the diagram commutes.
\end{proof}

\noindent For $\sigma \in \mathrm{Rep}_{n}$ and for $\chi \in \Irr$, let $d^{\sigma}_{\chi} := \mathrm{dim}\left(\Hom_{\Gamma}(X_{\chi},M^{\sigma})\right)$.
Denote the character map by
\begin{equation*}
\mathrm{char}: \begin{array}{ccc}
\RepGn & \to & \Delta_{\Gamma}\\
\sigma & \mapsto & d^{\sigma}
\end{array}.
\end{equation*}

\noindent Let $\tilde{\mathcal{A}}_n := \mathrm{char}(\RepGn)$ which is just a combinatorial way to encode $\mathrm{Char}_{n}$ the set of all characters associated to $n$-dimensional representations of $\Gamma$. Moreover for ${d \in \tilde{\mathcal{A}}_n}$, denote by $\mathcal{C}_d:=\mathrm{char}^{-1}(d)$ the set of all $n$-dimensional representations that have character $\sum_{\chi \in \Irr}{d_{\chi}\chi}$. Note that if we take $\sigma \in \mathcal{C}_d$, then $\mathcal{C}_d$ is just the $\mathrm{GL}_n(\C)$-conjugacy class of $\sigma$. With this notation, we have
\begin{equation*}
\RepGn = \coprod_{d \in \tilde{\mathcal{A}}_n}{\mathcal{C}_d}.
\end{equation*}

\begin{rmq}
The set $\tilde{\mathcal{A}}_n$ is equal to $\left\{d \in \Delta^+ \mid |d|=n\right\}$.
\end{rmq}

\noindent Denote by $\mathcal{A}_{n, \oando}:=\left\{d \in \tilde{\mathcal{A}}_n | q_{\oando}^{-1}(\mathcal{C}_d) \neq \emptyset\right\}$. For $d \in \mathcal{A}_{n, \oando}$, recall that we can associate an $\Irr$-graded vector space $V^d:=\bigoplus_{\chi \in \Irr}{\C^{d_{\chi}}}$. Denote by ${M^d:=\bigoplus_{\chi \in \Irr}{V^d_{\chi}\otimes X_{\chi}}}$ the $\Gamma$-module associated to $d$.

\begin{deff}
Take $d \in \mathcal{A}_{n, \oando}$. Let us define the variety $\mathcal{M}_d:=p_{\oando}\left(q_{\oando}^{-1}(\mathcal{C}_d)\right)$.
\end{deff}

\begin{rmq}
Note that $\mathcal{M}_d=p_{\oando}\left(q_{\oando}^{-1}(\sigma)\right)$ for any $\sigma \in \mathcal{C}_d$ since $\tilde{p}_{\oando}$ and $q_{\oando}$ are $\GL_n(\C)$-equivariant.\\
\end{rmq}

\begin{thm}
\label{decomp}
Let $(\theta,\lb) \in \mathbb{Q} \times \C \setminus \{(0,0)\}$.
For each integer $n$ and each finite subgroup $\Gamma$ of $\SL_2(\C)$, we have the following decomposition into irreducible components
\begin{equation*}
\QVJnb^{\Gamma} = \coprod_{d \in \mathcal{A}_{n, \oando}}{\QV_d}.
\end{equation*}
\end{thm}
\begin{proof}
Take $d \in \mathcal{A}_{n, \oando}$ and let us first show that $\QV_d$ is an irreducible and closed set of $\QVJnb^{\Gamma}$.
Take $\sigma \in \mathcal{C}_d$. Proposition \ref{prop_irr} gives that $\mathcal{M}_{\theta}^{\lb\delta}(M^{\sigma})$ is irreducible and since $\QV_d=\iota_{M^{\sigma}}\left(\mathcal{M}_{\theta}^{\lb\delta}(M^{\sigma})\right)$, we have that $\QV_d$ is irreducible.

To show that $\QV_d$ is a closed set of $\QVJnb^{\Gamma}$, note that $L_{\oando}$ is connected and that $\tilde{p}_{\oando}$ is injective thanks to Lemma \ref{inj}. The image of $p_{\oando}$ being $\QVJnb^{\Gamma}$, this implies that
\begin{equation*}
\tilde{p}_{\oando}\left(L_{\oando}\right)= \pi^{-1}\left(\QVJnb^{\Gamma}\right).
\end{equation*}
The group $\GL_n(\C)$ acts freely on $\mu_{\bullet}^{-1}(\lb)^{\theta-ss}$ (Remark \ref{free action}), which gives that $\pi$ is a smooth morphism. Moreover, the group $\Gamma$ being a finite group, we know that $\QVJnb^{\Gamma}$ is smooth and in particular $\pi^{-1}\left(\QVJnb^{\Gamma}\right)$ is smooth. Thus, the morphism $\tilde{p}_{\oando}$ becomes an isomorphism of algebraic varieties between $L_{\oando}$ and $\pi^{-1}\left(\QVJnb^{\Gamma}\right)$.
Let us denote by ${p'_{\oando}:\pi^{-1}\left(\QVJnb^{\Gamma}\right) \to L_{\oando}}$ its inverse. Consider now
\begin{center}
$\tilde{\tau}_{\oando}:=q_{\oando} \circ p'_{\oando} : \pi^{-1}\left(\QVJnb^{\Gamma}\right) \to \RepGn$.
\end{center}
Since $\tilde{\tau}_{\oando}$ is $\GL_n(\C)$-equivariant, we have $\tau_{\oando}: \QVJnb^{\Gamma} \to \RepGn\sslash\GL_n(\C)$.\\
Here is the big picture
\vspace{0.5cm}
\begin{center}
$\begin{tikzcd}[row sep=large, column sep=large]
L_{\oando} \ar[r, "q_{\oando}"] \ar[d, shift right, "\tilde{p}_{\oando}"'] \ar[dr, "p_{\oando}"'] & \RepGn \ar[r, two heads] & \RepGn\sslash \mathrm{GL}_n(\C)\\
\pi^{-1}\left(\QVJnb^{\Gamma}\right) \ar[u,shift right, "p'_{\oando}"'] \ar[r,"\pi"']& \QVJnb^{\Gamma} \ar[ur, "\tau_{\oando}"']
\end{tikzcd}$.
\end{center}
\vspace{0.5cm}

\noindent It is then clear that
\begin{equation*}
\tau_{\oando}^{-1}(\mathcal{C}_d) = p_{\oando}\left(q_{\oando}^{-1}(\mathcal{C}_d)\right) = \QV_d
\end{equation*}
which proves that $\QV_d$ is a closed set. Indeed $\mathcal{C}_d$ is closed because all representations of $\Gamma$ are semisimple since $\Gamma$ is a finite group.\\
Finally, we have to show that $\QVJnb^{\Gamma}=\bigcup_{d \in \mathcal{A}_{n,\oando}}{\QV_d}$ which comes for free
 \begin{equation*}
 \QVJnb^{\Gamma} = \bigcup_{d \in \mathcal{A}_{n,\oando}}{\tau_{\oando}^{-1}(\mathcal{C}_d}) = \bigcup_{d \in \mathcal{A}_{n,\oando}}{\QV_d}.
 \end{equation*}
\end{proof}

\begin{prop}
\label{prop_iso_quiver_var}
For each $d \in \mathcal{A}_{n, \oando}$ and for each $\sigma \in \mathcal{C}_d$
\begin{equation*}
\iota_{M^{\sigma}}:  \QV_{\theta}^{\lb\delta}(M^{\sigma}) \to \QV_d
\end{equation*}
is an isomorphism of algebraic varieties.
\end{prop}
\begin{proof}
The morphism $\iota_{M^{\sigma}}$ is injective. Using Proposition \ref{rep works} and \ref{prop_commutes}, we have
\begin{equation*}
\QV_d=\iota_{M^{\sigma}}\left(\QV_{\theta}^{\lb\delta}(M^{\sigma})\right).
\end{equation*}
Moreover, the variety $\QV_d$ is smooth. Indeed, thanks to Theorem \ref{decomp}, we have that $\QV_d$ is an irreducible component of the smooth variety $\QVJnb^{\Gamma}$. Furthermore, since $\QVJnb^{\Gamma}$ has a finite number of irreducible components, this implies that
 $\QV_d$ is open.
Finally, using Theorem \ref{link_var} and Proposition \ref{prop_irr}, we have that $\QV_{\theta}^{\lb\delta}(M^{\sigma})$ is connected. Summing it all up, we can now conclude that $\iota_{M^{\sigma}}$ is an isomorphism of algebraic varieties.
\end{proof}

\section{Irreducible components of ${\Hk_n^{\Gamma} \text{ and } \Ck_n^{\Gamma}}$}
\label{chap_irr}

In the previous section, we have studied the $\Gamma$-fixed point locus of the Jordan quiver variety $\QVJnb$ when $(\theta,\lb) \neq (0,0)$. In this section we will first use Theorem \ref{decomp} to retrieve the irreducible components of the Hilbert scheme of $n$ points in $\C^2$ and of the $n^{\text{th}}$ Calogero-Moser space. In a second part, a combinatorial description of the indexing set $\mathcal{A}_{n,\oando}$ will be given. To do so, we will work with the McKay realization (Remark \ref{rmq_mckay}) and a new statistic on the root lattice $\mathcal{Q}$ will be introduced.

\subsection{Hilbert scheme of $n$ points in $\C^2$}
On the one side, denote the Hilbert scheme of $n$ points on $\C^2$ by $\Hk_n$ which is
\begin{equation*}
\left\{I \subset \C[x,y]| I \text{ is an ideal and } \mathrm{dim}(\C[x,y]/I)=n\right\}.
\end{equation*}
John Fogarty showed \cite[Proposition $2.2$ \& Theorem $2.9$]{Fo68} that $\Hk_n$ is a smooth connected $2n$ dimensional algebraic variety. The algebraic group $\GL_2(\C)$ acts naturally on $\C^2$, thus on the coordinate ring $\C[x,y]$. This action induces a $\GL_2(\C)$-action on $\Hk_n$.
Our fixed group $\Gamma$ being a finite subgroup of $\SL_2(\C)$, $\Hk_n^{\Gamma}$, the locus of $\Gamma$-fixed elements of $\Hk_n$ is also a smooth algebraic variety. For $I \in \Hk_n^{\Gamma}$, the $n$-dimensional vector space $\C[x,y]/I$ is a representation of $\Gamma$. For $\chi \in \CharGn$, let us denote by  $\Hk_n^{\chi}:=\{ I \in \Hk_n^{\Gamma} | \mathrm{Tr}(\C[x,y]/I)=\chi\}$. On the level of sets, we have the following decomposition
\begin{equation*}
\Hk_n^{\Gamma}=\coprod_{\chi \in \CharGn}{\Hk_n^{\chi}}.
\end{equation*}
In what follows, we will show that this decomposition lifts to the category of algebraic varieties and that every nonempty $\Hk_n^{\Gamma,\chi}$ is an irreducible component of $\Hk_n^{\Gamma}$.
Let us construct a morphism between $\Hk_n$ and $\QVJn$. Given an ideal $I \in \Hk_n$ consider the following two linear maps
\begin{multicols}{2}

\begin{itemize}
\item $m_x^I:\begin{array}{ccc}
\C[x,y]/I &\to& \C[x,y]/I\\
\bar{P} &\mapsto& \overline{xP}\\
\end{array}$\\
\item $m_y^I:\begin{array}{ccc}
\C[x,y]/I &\to& \C[x,y]/I\\
\bar{P} &\mapsto& \overline{yP}\\
\end{array}$\\
\end{itemize}
\end{multicols}
\noindent Denote  $v^I=\bar{1} \in \C[x,y]/I$ and define
\begin{equation*}
H:
\begin{array}{ccc}
\Hk_n & \to & \QVJn\\
I & \mapsto & \overline{(m_x^{I},m_y^I, v^I, 0)}  \\
\end{array}.
\end{equation*}
\begin{prop}
\label{GL2-equiv}
The morphism $H$ is a $\GL_2(\C)$-equivariant isomorphism.
\end{prop}
\begin{proof}
The fact that $H$ is an isomorphism is proven in \cite[Theorem 11.5]{Kir06}. Let ${g=\begin{pmatrix}a & b \\c &d\end{pmatrix}}$ be an element of $\GL_2(\C)$ and $I \in \Hk_n$. Then
\begin{equation*}
\nu :\begin{array}{ccc}
\C[x,y]/I & \to & \C[x,y]/g.I\\
\bar{P} & \mapsto & \overline{g.P}  \\
\end{array}
\end{equation*}
is an isomorphism.
It is clear that $\nu^{-1}(v^{g.I})=v^I$. A simple computation gives
\begin{equation*}
\nu^{-1}\circ m_x^{g.I}\circ \nu=am_x^I+bm_y^I
\end{equation*}
\begin{equation*}
\nu^{-1}\circ m_y^{g.I}\circ\nu=cm_x^I+dm_y^I
\end{equation*}
which shows that $H(g.I)=g.H(I)$.\\
\end{proof}
\begin{rmq}
If $d \in \tilde{A}_{n}$, note that $\mathcal{M}_d$ is equal to $H(\Hk_n^{\xi_d})$ where ${\xi_d:=\sum_{\chi \in \Irr}{d_{\chi}\chi}}$ is an element of  $\CharGn$. We also have $|d|=k_{\xi_d}=n$.\\
\end{rmq}

\begin{cor}
\label{cor_hn}
For each integer $n$ and each finite subgroup $\Gamma$ of $\SL_2(\C)$, we can decompose the $\Gamma$-fixed point locus of the Hilbert scheme of $n$ points in $\C^2$ into irreducible components
\begin{equation*}
\Hk_n^{\Gamma} = \coprod_{d \in \mathcal{A}_{n,1,0}}{\Hk_n^{\xi_{d}}}.
\end{equation*}
\end{cor}
\begin{proof}
Applying $H^{-1}$ to Theorem \ref{decomp} and using Proposition \ref{GL2-equiv}, we get this decomposition into irreducible components.
\end{proof}

\subsection{Calogero-Moser space}
On the other side, i.e. when $(\theta,\lb)=(0,1)$, let us introduce the Calogero-Moser space. Denote the $n^{\text{th}}$ Calogero-Moser space by $\Ck_n$ which can be defined as
\begin{equation*}
\{(X,Y) \in \mathrm{M}_n(\C)^2| XY-YX + I_n \text{ is a rank 1 matrix}\}\sslash \mathrm{GL}_n(\C)
\end{equation*}
where $\GL_n(\C)$ acts by base change on $X$ and $Y$. George Wilson showed that $\Ck_n$ is a smooth, connected affine algebraic variety of dimension $2n$ \cite[Section 1]{Wil98}.
We define a natural action of $\GL_2(\C)$ on $\{(X,Y) \in \mathrm{M}_n(\C)^2| XY-YX + I_n \text{ is a rank 1 matrix}\}$
\begin{equation*}
g.(X,Y):= (aX+bY,cX+dY), \quad \forall g=\begin{pmatrix} a &b \\c & d \end{pmatrix} \in \GL_2(\C).
\end{equation*}
This action commutes with the $\GL_n(\C)$ and gives a well defined action on $\Ck_n$.\\
Consider now the morphism
\begin{center}
${C:\begin{array}{ccc}
\QVJndefu & \to & \Ck_n \\
\overline{(\alpha,\beta,v,w)}& \mapsto & \overline{(\alpha,\beta)}  \\
\end{array}}$.
\end{center} By construction, we have the following Lemma.\\
\begin{lemme}
\label{C_equi}
The isomorphism $C$ is an isomorphism of $\GL_2(\C)$-varieties.\\
\end{lemme}

\noindent For $d\in \tilde{\mathcal{A}}_{n}$, let us denote by $\Ck_{d}:=C(\QV_d)$.\\
\begin{cor}
For each integer $n$ and each finite subgroup $\Gamma$ of $\SL_2(\C)$, we can decompose the $\Gamma$-fixed point locus of the $n^{\text{th}}$ Calogero-Moser  space into irreducible components
\begin{equation*}
\Ck_n^{\Gamma} = \coprod_{d \in \mathcal{A}_{n,0,1}}{\Ck_{d}}.
\end{equation*}
\end{cor}
\begin{proof}
Apply $C$ to Theorem \ref{decomp} and use Lemma \ref{C_equi} to obtain this decomposition.
\end{proof}

\subsection{On the parametrization sets $\mathcal{A}_{n,\oando}$}

Fix $(\theta,\lambda) \neq (0,0)$ and let us now give a combinatorial model of the indexing set $\mathcal{A}_{n,\oando}$.
Recall that for $a\in \mathcal{Q}_{\mathrm{fin}}$ we have denoted by $t_a \in W$ the image of $a$ by the isomorphism $W_{\mathrm{fin}} \ltimes \mathcal{Q}_{\mathrm{fin}} \iso W$. In the following, we will denote the fact that there exists $k \in \mathbb{Z}$, such that $a-b= kc$ for $(a,b,c) \in \left(\mathfrak{h}^*\right)^3$, by $a \equiv b [c]$.\\
\begin{lemme}
\label{trans}
For all $(a,d) \in \mathcal{Q}_{\mathrm{fin}} \times \mathcal{Q}$, we have $t_a.d \equiv d- a \text{ }[\delta]$.
\end{lemme}
\begin{proof}
Thanks to relation (\ref{link_eq}) and \cite[Formula $6.5.2$]{kac90}, we have
\begin{equation*}
t_a.d = \Lambda_0 - t_a*(\Lambda_0-d) \equiv d- a + \langle d, \delta^{\vee} \rangle [\delta].
\end{equation*}
Since $d \in \mathcal{Q}$, $\langle d,\delta^{\vee}\rangle=0$ by definition of $\delta^{\vee}$.
\end{proof}

\begin{lemme}
\label{lemme_ld}
For each $d \in \mathcal{Q}$, there exists a unique integer $r$ such that $d$ and $r\delta$ are in the same $W$-orbit for the $.$ action from Definition \ref{def_action}.
\end{lemme}
\begin{proof}
Take $d \in \mathcal{Q}$, then $a:=d-d_0\delta \in \mathcal{Q}_{\mathrm{fin}}$ and thanks to Lemma \ref{trans}, $t_a.d$ is an element of the desired form.

Now suppose  that there are two integers  $r_1$ and $r_2$ such that $r_1\delta$ and $r_2\delta$ are in the same $W$-orbit. Since $\delta$ is in the kernel of the generalized Cartan matrix of type $\Gamma$, $\delta$ is fixed under the action of $W_{\mathrm{fin}}$. This observation reduces the $W$-orbit of $r_1\delta$ to the $\mathcal{Q}_{\mathrm{fin}}$-orbit. There must then exist $a \in \mathcal{Q}_{\mathrm{fin}}$ such that $t_a.r_1\delta=r_2\delta$. Using Lemma \ref{trans}, $t_a.r_1\delta \equiv r_1\delta-a\text{ } [\delta]$, we can conclude that $a=0$ and that $r_1=r_2$.
\end{proof}

\begin{deff}
\label{deff_weight}
For $d \in \mathcal{Q}$ let the \textbf{weight of} $d$ be the unique integer $r$ such that $r\delta$ and $d$ are in the same $W$-orbit. It will be denoted by $\mathrm{wt}(d)$.\\
\end{deff}

\noindent The following proposition establishes a bridge between quiver varieties and combinatorics.\\
\begin{prop}
\label{link geom comb}
Let $(\theta,\lb) \in \Theta^{++}\times \Lambda^{++}$. For each $d \in \Delta$
\begin{equation*}
 \QV_{\theta}(d) \neq \emptyset \iff \mathrm{wt}(d) \geq 0 \iff \QV^{\lb}(d) \neq \emptyset.
\end{equation*}

\end{prop}
\begin{proof}
Using \cite[Theorem $10.2$]{Nak98} (or  \cite[Theorem $13.19$]{Kir06}), the variety $\QV_{\theta}(d)$ is nonempty if and only if $\Lambda_0-d$ is a weight of the basic representation $L(\Lambda_0)$ of the Kac-Moody Lie algebra of type $\Gamma$. Moreover, using \cite[Theorem $20.23$]{Cart05} we have that the set of weights of $L(\Lambda_0)$ is the following
\begin{equation*}
\left\{\Lambda_0+\gamma-(\frac{1}{2}\langle \gamma, \gamma \rangle - k) \delta \mid \gamma \in \mathcal{Q}_{\mathrm{fin}}, k \in \mathbb{Z}_{\geq 0}\right\}.
\end{equation*}
The Lemma \ref{lemme_ld}, tells us that there is $\omega \in W$ such that $d=\omega.(r\delta)$ with $r=\mathrm{wt}(d)$. Since $\delta$ is fixed by the action of  $W_{\mathrm{fin}}$, there is $a \in Q_{\mathrm{fin}}$ such that $d=t_a.(r\delta)$. Thanks to \cite[Formula $6.5.2$]{kac90} and the definitions of $\delta$ and $\Lambda_0$, we have that
\begin{equation*}
t_a*(r\delta)= r\delta -a + \frac{1}{2}\langle a, a \rangle \delta.
\end{equation*}
Using relation (\ref{link_eq}), we can conclude that $d$ is a weight of $L(\Lambda_0)$ if and only if $r \geq 0$.
In addition, we know that both $\QV_{2\lb}(d)$ and $\QV^{\lb}(d)$ are hyper-K\"ahler reductions \cite[Corollary $6.2$]{Ki94} and \cite[Theorem $3.1$]{Nak94}. Using the rotation map defined in \cite[Section $3.7$]{Gor08}, we have that these varieties are diffeomorphic. By the first equivalence, we have that $\QV^{\lb}(d)$ is nonempty if and only if $\mathrm{wt}(d) \geq 0$.
\end{proof}

\begin{thm}
\label{thm_indexing}
For each integer $n$ and each finite subgroup $\Gamma$ of $\SL_2(\C)$, the set indexing the irreducible components of $\Hk_n^{\Gamma}$ and of $\Ck_n^{\Gamma}$ are equal
\begin{equation*}
\mathcal{A}_{n,1,0}=\mathcal{A}_{n,0,1} = \left\{d \in \Delta^+  \mid |d|=n, \mathrm{wt}(d) \geq 0\right\}.
\end{equation*}
\end{thm}
\begin{proof}
Take $d \in \tilde{\mathcal{A}_{n}} \subset \Delta_{\Gamma}$ and let us reformulate combinatorially the definiton of  $\mathcal{A}_{n}$. Thanks to the proof of Theorem \ref{decomp}, $d$ is in $\mathcal{A}_{n,\theta,\lb}$ if and only if  $\QV_d$ is nonempty. Since ${\QV_d=\iota_{M^{\sigma}}\left(\QV_{\theta}^{\lb\delta}(M^{\sigma})\right)}$ for any $\sigma \in \mathcal{C}_d$, it is nonempty if and only if $\QV_{\theta}^{\lb\delta}(d)$ is nonempty. Applying Proposition \ref{link geom comb} first for $\theta=1$ and then for $\lb=\delta$, gives the result.
\end{proof}

\noindent From now on let us denote $\mathcal{A}_{n,1,0}=\mathcal{A}_{n,0,1}$ by $\mathcal{A}_{n}$.\\

\begin{rmq}
Note that when $\Gamma$ is the cyclic group $\mu_{\ell}$ of order $\ell$ in $\mathbb{T}_1$, the maximal diagonal torus of $\SL_2(\C)$, an explicit combinatorial model of $\mathcal{A}_n$ in terms of $\ell$-cores exists \cite[Lemma $7.8$]{Gor08}. Note also that Theorem \ref{thm_indexing} combined with \cite[Lemma $4.9$]{BM21} is another way to obtain this model.
\end{rmq}

\noindent To finish, let us give a simple expression of the dimension of the irreducible components $\mathcal{M}_d$ for each $d \in \mathcal{A}_n$.\\

\begin{prop}
\label{prop_dim_irr_comp}
For each $d \in \mathcal{A}_n$, the variety $\mathcal{M}_d$ has dimension $2\mathrm{wt}(d)$.
\end{prop}
\begin{proof}
Take $d \in \mathcal{A}_n$. There exists $\omega_d \in W$  and $r_d \in \mathbb{Z}_{\geq 0}$ such that $\omega_d.d=r_d\delta$. If $r_d=0$, it is clear that the dimension of $\mathcal{M}_d$ is $0$. Now if $r_d >0$, Proposition \ref{prop_iso_quiver_var} gives that $\mathcal{M}_d$ is isomorphic to $\mathcal{M}_{1}(M^{\sigma})$ for $\sigma \in \mathcal{C}_d$. Moreover, Theorem \ref{link_var} gives that $\mathcal{M}_{1}(M^{\sigma})$ is isomorphic to $\mathcal{M}_{1}(d):=\mathcal{M}_{1}(d,\chi_0)$. We can now use Proposition \ref{maff} and Proposition \ref{dim_r_delt}, to conclude that the dimension is $2r_d$.
\end{proof}

\section{Resolutions of $\mathcal{Y}_k$ and $\pi_0(\Hk_n^{\Gamma})$}
\label{sect_res}

Recall from introduction that with $\Gamma$ and $n$ we have build a group $\Gamman= \mathfrak{S}_n \ltimes \Gamma^n$ and a symplectic singularity $\mathcal{Y}_n=\C^{2n}/\Gamman$. Gwyn Bellamy and Alastair Craw have classified all projective, symplectic resolutions of $\mathcal{Y}_n$ in terms of quiver varieties \cite[Corollary $1.3$]{BC18}. Moreover we have shown, in section \ref{chap_irr}, that for each integer $k\geq 1$, the irreducible components of $\Hk_k^{\Gamma}$ can be described in terms of quiver varieties. A natural question then arises.
Given $X \to \mathcal{Y}_n$ a projective, symplectic resolution, is it possible to find an integer $p_X$ and an irreducible component of $\Hk_{p_X}^{\Gamma}$ that is isomorphic to $X$ over $\mathcal{Y}_n$ ?  \\
To answer this question, this section is decomposed into two subsections. In the first subsection, we recall the description of all projective, symplectic resolutions of $\C^{2n}/\Gamman$ done by Gwyn Bellamy and Alastair Craw. In the second subsection, we will explain how to describe these resolutions as irreducible components of the $\Gamma$-fixed point locus of the Hilbert scheme of points in the plane.

\subsection{Chamber decomposition inside $\Theta$}
In this section, we will make use of the $\R$-vector space $\Theta^{\mathbb{R}}:=\Hom_{\mathbb{Z}}(\Delta,\mathbb{R})$. Let us first recall  the notation used in \cite{BC18}. Let $F$ be the following simplicial cone in $\Theta^{\R}$
\begin{equation*}
\left\{\theta \in \Theta^{\R} \mid \theta(\delta) \geq 0, \forall \chi \in \Irr\setminus\{\chi_0\}, \theta(\chi) \geq 0 \right\}.
\end{equation*}
For $d \in \Delta$, let us denote $d^{\perp}:=\{\theta \in \Theta^{\R} \mid \theta(d)=0\}$. Consider the following set of walls in $\Theta^{\R}$
\begin{equation*}
\mathcal{W}_n:=\{\delta^{\perp}\} \cup \left\{(m\delta+\alpha)^{\perp} \mid \alpha \in \Phi_{\mathrm{fin}}, -n < m < n\right\}.
\end{equation*}
\begin{deff}
A connected component $\Ch$ of $\big(\Theta^{\R} \big) \setminus \bigcup_{c^{\perp} \in \mathcal{W}_n}{c^{\perp}}$ will be called a GIT chamber of $\Theta_{\Gamma}$. Let us denote by $\Theta^{\mathrm{reg}}$ the union of all GIT chambers of $\Theta^{\R}$.\\
\end{deff}

\begin{rmq}
Let $F^{\mathrm{reg}}:=\Theta^{\mathrm{reg}}\cap F$ and note that by construction of $F$, $F^{\mathrm{reg}}$ is a union of GIT chambers.\\
\end{rmq}

\noindent We can now reformulate the main result \cite[Corollary 6.4]{BC18} as follows.\\

\begin{thm}[\cite{BC18}]
\label{thm_BC}
For each projective, symplectic resolution $X \to \mathcal{Y}_n$ there exists a unique GIT chamber $\Ch$ in $F$ such that $X$ is isomorphic to the quiver variety $\mathcal{M}_{\theta}(n\delta)$ over $\mathcal{Y}_n$ for any $\theta \in \Ch \cap \Theta$.
\end{thm}


\noindent Let us now come back to what has been done in section \ref{chap_irr} and to the realization described at Remark \ref{rmq_mckay}. This is inspired by \cite[Section $6.6$]{kac90}. Let us fix a real form $\mathfrak{h}^{\R}$ of $\mathfrak{h}$ that contains $\{\Lambda_{\alpha_0}^{\vee}\} \cup \{\tilde{\alpha}_{\chi} | \chi \in \Irr\}$ and such that $\forall \chi \in \Irr, \alpha_{\chi }(\mathfrak{h}^{\R}) \subset \R$. Denote by $\mathcal{V}$ the quotient space $\mathfrak{h}^{\R}/\R\delta^{\vee}$. By definition of $\delta^{\vee}$, $\langle \delta, \delta^{\vee} \rangle$ is equal to 0. We can then consider ${E:=\{h \in \mathcal{V} | \langle \delta, h \rangle=1\}}$. Let $E^0:=\{h \in \mathcal{V} | \langle \delta,h \rangle=0\}$. Then $(E,E^0)$ is an affine space in $\mathcal{V}$. 
It is clear that  $E^0=\mathrm{Vect}\left(\{\tilde{\alpha}_{\chi} \mid \chi \in \Irr\}\right)$ but the family $\{\tilde{\alpha}_{\chi}\mid \chi \in \Irr\}$ is no longer linearly independant since $\sum_{\chi \in \Irr}{\delta_{\chi}\tilde{\alpha}_{\chi}}=0$. By definition,  we have $\Lambda^{\vee}_0:=\Lambda^{\vee}_{\alpha_0} \in E$. Denote $\tau \in \mathcal{Q}_{\mathrm{fin}}$ and $\tau^{\vee} \in \mathcal{Q}^{\vee}_{\mathrm{fin}}$ respectively the highest root and coroot of the finite type associated with $\Gamma$. Let us abuse notation and denote $\mathcal{Q}^{\vee}_{\mathrm{fin}}:=\mathcal{Q}^{\vee}_{\mathrm{fin}}\otimes \R$. Consider
\begin{center}
$\mathrm{Aff}^0:\begin{array}{ccc}
E^0 & \to & \mathcal{Q}^{\vee}_{\mathrm{fin}}\\
\tilde{\alpha}_{\chi} & \mapsto &
\begin{cases}
\tilde{\alpha}_{\chi} & \text{ if } \chi \neq \chi_0,\\
-\tau^{\vee} & \text{ if } \chi=\chi_0.
\end{cases}\\
\end{array}$
\end{center}
\vspace*{0.25cm}
This linear map is well-defined since $\delta^{\vee}=\tilde{\alpha}_{\chi_0}+\tau^{\vee}$ and it then induces an affine map ${\mathrm{Aff}: (E,E^0) \to (\mathcal{Q}^{\vee}_{\mathrm{fin}},\mathcal{Q}^{\vee}_{\mathrm{fin}})}$ such that $\mathrm{Aff}(\Lambda_{0}^{\vee})=0$. The linear map $\mathrm{Aff}^0$ is surjective and by dimension is then an isomorphism. This implies that $\mathrm{Aff}$ is an isomorphism. Recall that in section \ref{chap_irr}, the natural $W$-action by reflections on $\mathfrak{h}$ has been denoted $*$. Since $\forall \chi \in \Irr, s_{\chi}*\delta^{\vee}=\delta^{\vee}$, let us equip $\mathcal{V}$ with this $W$-action.\\

\begin{lemme}
The set $E$ is  $W$-stable.
\end{lemme}
\begin{proof}
Let us take $v \in E$. It is then enough to show that $\forall \chi \in \Irr, \langle \delta, s_{\chi}*v\rangle=1$. By definition of $\delta$: $\forall \chi \in \Irr, \langle \delta,\tilde{\alpha}_{\chi}\rangle=0$.
We then have that $s_\chi*v \in E$.
\end{proof}

\begin{prop}
The induced action of $W$ on $\mathcal{Q}^{\vee}_{\mathrm{fin}}$ via $\mathrm{Aff}$ is the usual action of the affine Weyl group defined in \cite[Chap. VI, paragraph 2, no. 1]{Bourblie}.
\end{prop}
\begin{proof}
It is enough to check that $s_{\chi_0}$ acts on $\mathcal{Q}^{\vee}_{\mathrm{fin}}$ as $t_{\tau^{\vee}}s_{\tau}$. Indeed, the element $s_{\chi}$ acts naturally as an element of the finite Weyl group $W_{\mathrm{fin}}$, for each ${\chi \in \Irr\setminus \{\chi_0\}}$. Take $\chi \in \Irr$, we have
\begin{align*}
\mathrm{Aff}(s_0*(\Lambda^{\vee}_0+\tilde{\alpha}_{\chi})) = & \mathrm{Aff}(\Lambda^{\vee}_0+\tilde{\alpha}_{\chi}-\tilde{\alpha}_{\chi_0}- \langle \alpha_{\chi_0},\tilde{\alpha}_{\chi} \rangle \tilde{\alpha}_{\chi_0}) \\
                                                          = & \tilde{\alpha}_{\chi} + \tau^{\vee}+ \langle  \alpha_{\chi_0},  \tilde{\alpha}_{\chi} \rangle \tau^{\vee}\\
                                                          = & \tilde{\alpha}_{\chi} - \langle \tau, \tilde{\alpha}_{\chi} \rangle \tau^{\vee} + \tau^{\vee}.
\end{align*}
The last equality comes from the fact that $\alpha_{\chi_0}=\delta-\tau$.
\end{proof}

\noindent For each $\alpha \in \Phi_{\mathrm{fin}}$ and each $k \in \mathbb{Z}$, $L_{\alpha,k}:=\{x \in \mathcal{Q}^{\vee}_{\mathrm{fin}} \mid \langle \alpha,x \rangle = k\}$ defines a hyperplane in $\mathcal{Q}^{\vee}_{\mathrm{fin}}$. Let us denote by $\mathcal{W}_{\mathrm{aff}}:=\{\mathrm{Aff}^{-1}(L_{\alpha,k})| (\alpha,k) \in \Phi_{\mathrm{fin}} \times \mathbb{Z} \}$. We get an affine hyperplane arrangement in $E$.\\ Denote by
\begin{center}
$\kappa^{\vee}_{\R}: \begin{array}{ccc}
\mathcal{V} & \to & \Theta^{\R}\\
x & \mapsto & (\chi \mapsto \langle \alpha_{\chi},x \rangle) \\
\end{array}$.
\end{center}
\vspace*{0.25cm}
We then have that ${E\subset {\kappa^{\vee}_{\R}}^{-1}(\Theta^{\R}) = \mathcal{V}}$.
\begin{prop}
\label{hyp_ok}
For each hyperplane $\gamma^{\perp} \in \mathcal{W}_n$, ${\kappa^{\vee}_{\R}}^{-1}(\gamma^{\perp})\cap E \in \mathcal{W}_{\mathrm{aff}}$.
\end{prop}
\begin{proof}
The set ${\kappa^{\vee}_{\R}}^{-1}(\delta^{\perp})$ is equal to $E^0$, which implies that ${\kappa^{\vee}_{\R}}^{-1}(\delta^{\perp})\cap E = \emptyset$. If $\alpha \in \Phi_{\mathrm{fin}}$ and $m \in \llbracket 1-n, n-1 \rrbracket$, then
\begin{equation*}
{\kappa^{\vee}_{\R}}^{-1}((m\delta+\alpha)^{\perp}) \cap E =\{v \in E | \langle \alpha, v \rangle = -m \} = \mathrm{Aff}^{-1}(L_{\alpha,-m}) \in \mathcal{W}_{\mathrm{aff}}.
\end{equation*}
\end{proof}

\begin{rmq}
\label{link_chamb}
This Proposition implies that for each GIT chamber $\Ch \subset \Theta^{\mathrm{reg}}$, ${{\kappa^{\vee}_{\R}}^{-1}(\Ch) \cap E}$ is a union of alcoves.\\
\end{rmq}

\noindent Let us now restrict our attention to the cone $F$.\\

\begin{prop}
\label{cone_F}
$(i)$ If $\mathcal{C}_{f}:=\{v \in \mathcal{Q}^{\vee}_{\mathrm{fin}} \mid \forall  \chi \in \Irr \setminus \{\chi_0\}, \langle \alpha_{\chi},v \rangle >0 \}$ denotes the fundamental Weyl chamber in $\mathcal{Q}^{\vee}_{\mathrm{fin}}$, then $\mathrm{Aff}({\kappa^{\vee}_{\R}}^{-1}(F)\cap E)=\overline{\mathcal{C}_f}$.\\
$(ii)$ If $\Ch \subset F^{\mathrm{reg}}$ is a GIT chamber such that ${\kappa^{\vee}_{\R}}^{-1}(\Ch) \cap E$ is bounded then ${\kappa^{\vee}_{\R}}^{-1}(\Ch) \cap E$  is an alcove in $E$.
\end{prop}
\begin{proof}
The first statement follows directly from the definition of $\mathcal{C}_f$ and $F$. Let us now prove the second statement. Take $\Ch$ such a chamber. Use Proposition \ref{hyp_ok} to prove that ${\kappa^{\vee}_{\R}}^{-1}(\Ch) \cap E$ is a union of alcoves. The definition of $\mathcal{W}_{\mathrm{aff}}$ and the fact that ${\kappa^{\vee}_{\R}}^{-1}(\Ch) \cap E$ is bounded imply that ${\kappa^{\vee}_{\R}}^{-1}(\Ch) \cap E$ is equal to exactly one alcove of $E$.
\end{proof}

\subsection{Resolutions as irreducible components}

Recall that in \cite[Example 2.1]{BC18}, authors have introduced the following GIT chamber
\begin{equation*}
\Ch_+=\{\theta \in \Theta^{\R} \mid \forall \chi \in \Irr, \theta(\chi)>0 \}.
\end{equation*}
\begin{rmq}
It is clear that $\Ch_+ \subset F$. Note moreover that ${\kappa^{\vee}_{\R}}^{-1}(\Ch_+) \cap E$ is bounded and is the fundamental alcove  ${\Alf_+:=\{h \in E \mid \forall \chi \in \Irr \setminus \{\chi_0\},\langle \alpha_{\chi},h \rangle > 0 \text{ and } \langle \tau,h \rangle < 1\}}$.\\
\end{rmq}

\begin{lemme}
\label{1_ok}
The vector $1=\kappa^{\vee}\big(\overline{\sum_{\chi \in \Irr}{\Lambda_{\chi}^{\vee}}}\big)$ is in $\Ch_+$.
\end{lemme}
\begin{proof}
By definition $\forall \chi \in \Irr,1(\chi)=1 > 0$.
\end{proof}

\noindent The following proposition gives an isomorphism between quiver varieties for stability parameters in the GIT chamber $\Ch_+$.\\

\begin{prop}
\label{prop_c+}
For each $(\theta,d) \in \Ch_+ \times \Delta^+$, $\mathcal{M}_{1}(d) \simeq \mathcal{M}_{\theta}(d)$.
\end{prop}
\begin{proof}
Using Lemma \ref{1_ok}, we know that $1 \in \Ch_+$. Use now \cite[Theorem $3.3.2$]{DH98} and \cite[Theorem $5.6$]{Thad96} to obtain the result.
\end{proof}

\begin{rmq}
\label{rmq_theta_0}
Note that $\overline{\sum_{\chi \in \Irr}{\Lambda_{\chi}^{\vee}}}$ is not in general in $E$. From now on, let us fix ${\theta_0 \in \Ch_+\cap \Theta}$ such that ${\kappa^{\vee}}^{-1}(\theta_0)$ is in an alcove $\Alf_0 \subset E$. Thanks to Proposition \ref{prop_c+}, we have  $\mathcal{M}_{1}(d) \simeq \mathcal{M}_{\theta_0}(d)$ for each $d \in \Delta^+$.
\end{rmq}

\noindent For a given integer $k$, let us denote the set of all  isomorphism classes of projective symplectic resolutions of $\mathcal{Y}_k$ by $\Resol_k$ and by $\IC_k$ the set of all irreducible components of $\Hk_k$. Moreover, let $\Resol := \bigcup_{k \in \mathbb{N}}{\Resol_k}$ and $\IC :=\bigcup_{k \in \mathbb{N}}{\IC_k}$.

Take now an integer $k$ and  $I \in \IC_k$. Using Corollary \ref{cor_hn}, we know that there is an element $d \in \mathcal{A}_k$ such that $I=\Hk_k^{\zeta_d}$ and using Theorem \ref{thm_indexing}, we have a unique element $r_d\in \mathbb{Z}_{\geq 0}$ and an element $\omega_d$ of $W$ such that ${\omega_d.d=r_d\delta}$. The image by $\kappa^{\vee}_{\R}$ of the alcove $\omega_d*\mathcal\Alf_0$ is then contained in a GIT chamber $\Ch_d$. Thus $\QV_{\omega_d.\theta_0}(r_d\delta)$ is a projective symplectic resolution of $\mathcal{Y}_{r_d}$ (Theorem \ref{thm_BC}). Thanks to the isomorphism $\mathrm{Maff}$ and Remark \ref{rmq_theta_0}, we have that ${\QV_{\omega_d.\theta_0}(r_d\delta) \simeq \QV_{\omega_d.1}(r_d\delta)}$. The following map is then well defined
\begin{center}
$\mathcal{BC}: \begin{array}{ccc}
\IC & \to & \Resol \\
I(d):=\Hk_{|d|_{\Gamma}}^{\Gamma,\zeta_d} & \mapsto & \mathcal{M}_{\boldsymbol{\omega_d.1}}^{\Gamma}(r_d\delta) \\
\end{array}$.\\
\end{center}
\vspace*{0.25cm}

\begin{rmq}
Note that $\mathcal{BC}$ does not depend on the choice of $\omega_d$. Indeed, if $\omega'_d \in W$ such that $\omega'_d.d=r_d\delta$, then there exists $\omega_0 \in W_{\mathrm{fin}}$ such that $\omega'_d=\omega_0\omega_d$. Using Proposition \ref{losev}, we have $\mathrm{Maff}:\QV_{\omega_d.1}(r_d\delta) \iso \QV_{\omega'_d.1}(r_d\delta)$ over $\mathcal{Y}_{r_d}$.\\
\end{rmq}

\begin{thm}
\label{surj}
The map $\mathcal{BC}$ is surjective.
\end{thm}
\begin{proof}
Take a GIT chamber $\Ch\in F^{\mathrm{reg}}$. Using Remark \ref{link_chamb}, we can choose an alcove ${\Alf_{\Ch} \subset E}$ such that $\kappa^{\vee}_{\R}(\Alf_{\Ch}) \subset \Ch$. For each $\theta \in \kappa^{\vee}_{\R}(\Alf_{\Ch})\cap \Theta$, the variety $\mathcal{M}_{\theta}(n\delta)$ is then a projective, symplectic resolution of $\mathcal{Y}_n$. Thanks to Theorem \ref{thm_BC}, it is enough to show that there exists an irreducible component $\mathcal{I}_{\Ch,n} \in \mathcal{IC}$ such that ${\mathcal{BC}(\mathcal{I}_{\Ch,n})=\mathcal{M}_{\theta}(n\delta)}$ for an element ${\theta \in \kappa^{\vee}_{\R}(\Alf_{\Ch})\cap \Theta}$. The action of $W$ on alcoves in $E$ is transitive \cite[Chap. V, paragraph 3, no. 2, Th 1]{Bourblie}. There exists ${\omega_{\Ch} \in W}$ such that ${\omega_{\Ch}*{\kappa^{\vee}}^{-1}(\theta_0) \in \Alf_{\Ch}}$. The map $\kappa^{\vee}$ being $W$-equivariant (Remark \ref{rmq_k_equi}), $\omega_{\Ch}.\theta_0 \in \kappa^{\vee}_{\R}(\Alf_{\Ch})\cap \Theta$. Using the isomorphism $\mathrm{Maff}$, we have
\begin{equation*}
\mathcal{M}_{\omega_{\Ch}.\theta_0}(n\delta) \simeq \mathcal{M}_{\theta_0}\left(\omega_{\Ch}^{-1}.(n\delta)\right) \simeq \mathcal{M}_{1}\left(\omega_{\Ch}^{-1}.(n\delta)\right).
\end{equation*}
The second isomorphism comes from Proposition \ref{prop_c+}.
Let $p_{n,\omega_{\Ch}}$ denote the integer $|\omega_{\Ch}^{-1}.(n\delta)|$ and take $\mathcal{I}_{\Ch,n}:=H^{-1}\left(\iota_{M^{\sigma}}\left(\mathcal{M}_{1}(M^{\sigma})\right)\right) \subset \Hk_{p_{n,\omega_{\Ch}}}^{\Gamma}$ which is an irreducible component for any $\sigma \in \mathcal{C}_{\omega_{\Ch}^{-1}.(n\delta)}$.
\end{proof}

\noindent The surjectivity directly implies the theorem stated in the introduction.\\
\begin{cor}
For each projective, symplectic resolution $X \to \mathcal{Y}_k$, there exists $n \in \Z_{\geq 1}$ and there exists an irreducible component $\mathcal{I}$ of $\Hk_{n}^{\Gamma}$ such that $X \iso \mathcal{I}$ over $\mathcal{Y}_k$.\\
\end{cor}

\noindent Let us finish by giving an example where $\mathcal{BC}$ is not injective. Take $n$ to be equal to $2$ and $\Gamma$ to be $\mu_3$, the cyclic group of order $3$. This group is generated by $\omega_3$ the diagonal matrix ${\mathrm{diag}(\zeta_3,\zeta_3^{-1}) \in \SL_2(\C)}$ where $\zeta_3$ is the primitive root of unity $e^{\frac{2i\pi}{3}}$. In that case, there will be $5$ GIT chambers in $F$.
Consider now $\omega_1=s_0s_2s_1s_2$ and $\omega_2=\omega_1s_0$ two elements of $W$. We have that $\omega_1.\theta_0$ and $\omega_2.\theta_0$ are in the same GIT chamber
\begin{equation*}
\Ch_{-}:=\{\theta \in \Theta^{\R} \mid \theta(\delta)>0, \forall (k,\alpha) \in \llbracket1-n,n-1 \rrbracket \times \Phi^+_{\mathrm{fin}},\theta(\alpha +m\delta)>0\}
\end{equation*}
Moreover since $\omega_1^{-1}.(2\delta) \neq \omega_2^{-1}.(2\delta)$, this shows that $I(\omega_1^{-1}.(2\delta)) \neq I(\omega_2^{-1}.(2\delta)$ but these two irreducible components have the same images under $\mathcal{BC}$.

\vspace*{0.5cm}

\noindent The set $\mathcal{A}_{n}$ has already been described in terms of $\ell$-cores when $\Gamma$ is the cyclic group $\mu_{\ell}$ of order $\ell$ contained in $\mathbb{T}_1$. In a future article we will study this set when $\Gamma$ is the binary dihedral group (type $D$) and prove that there are no interesting combinatorial models that can be constructed out of the $\mathbb{T}_1$-fixed points when $\Gamma$ is of the binary tetrahedral group ($E_6$), the binary octahedral group ($E_7$) and the binary icosahedral group ($E_8$).

\section{Appendix}
\label{appendix}
To improve readability, we postpone the proofs of the Propositions \ref{equiv cat} and \ref{link_momentum} to the end of the article.

\begin{proof}[Proof of Proposition \ref{equiv cat}]
\noindent Define two natural transformations of functors
\begin{multicols}{2}

\begin{itemize}
\item $\epsilon: \mathcal{F}\mathcal{G} \to \mathrm{Id}_{\RepGa}$,
\item  $\eta: \mathrm{Id}_{\Repm} \to \mathcal{G}\mathcal{F}$.
\end{itemize}
\end{multicols}
Start with $\epsilon$. Let us fix $(M,M^f,\Delta,Z_1,Z_2) \in \RepGa$, and consider
\begin{itemize}
\item $\epsilon_M: \bigoplus_{\chi \in \Irr}{\Hom_{\Gamma}(X_{\chi},M)\otimes X_{\chi}} \iso M$
\item $\epsilon_{M^f}: \bigoplus_{\chi \in \Irr}{\Hom_{\Gamma}(X_{\chi},M^f)\otimes X_{\chi}} \iso M^f$
\item $\epsilon_{\Delta}: \mathcal{F}(\mathcal{G}(\Delta)) \mapsto \epsilon_M\circ \mathcal{F}(\mathcal{G}(\Delta))\circ (\mathrm{id}_{\Xstd}\otimes\epsilon_M^{-1})$
\item $\epsilon_{Z_1}: \mathcal{F}(\mathcal{G}(Z_1) \mapsto \epsilon_{M} \circ \mathcal{F}(\mathcal{G}(Z_1)) \circ \epsilon_{M^f}^{-1} $
\item $\epsilon_{Z_2}: \mathcal{F}(\mathcal{G}(Z_2) \mapsto \epsilon_{M^f} \circ \mathcal{F}(\mathcal{G}(Z_2)) \circ \epsilon_{M}^{-1} $.\\
\end{itemize}

\noindent We need to show that $\epsilon_{\Delta}(\mathcal{F}(\mathcal{G}(\Delta)))=\Delta$ which can be reformulated as the following equality $\epsilon_M\circ \mathcal{F}(\mathcal{G}(\Delta)) =\Delta \circ (\mathrm{id}_{\Xstd}\otimes\epsilon_M)$.\\
Let us fix $\chi \in \Irr$ and  ${(a,f,z) \in \Xstd \times \Hom_{\Gamma}(X_{\chi},M) \times X_{\chi}}$, then
\begin{equation*}
\epsilon_M(\mathcal{F}(\mathcal{G}(\Delta))(a\otimes f \otimes z))=\Delta(f(\sum_{h \in \overline{E}, h'=\chi}{\tilde{y}_h(y_h(a\otimes z))})).
\end{equation*}
Thanks to Lemma \ref{sum works}, we have that $\Delta(f(\sum_{h \in \overline{E}, h'=\chi}{\tilde{y}_h(y_h(a\otimes z))}))=\Delta(f(a\otimes z))$ which is equal to $\left(\Delta \circ (\mathrm{id}_{\Xstd}\otimes\epsilon_M)\right)(a\otimes f \otimes z)$.\\
Furthermore, we have by construction that
\begin{equation*}
\epsilon_{M} \circ \mathcal{F}(\mathcal{G}(Z_1)) = Z_1 \circ \epsilon_{M^f}
\end{equation*}

\begin{equation*}
\epsilon_{M^f} \circ \mathcal{F}(\mathcal{G}(Z_2)) = Z_2 \circ \epsilon_{M}
\end{equation*}
Take now $(V,V^f,x,v^1,v^2) \in \Repm$. Define $\eta$\\

$\bullet$ $\eta_V: V \iso \bigoplus_{\chi \in \Irr}{\Hom_{\Gamma}(X_{\chi}, \bigoplus_{\xi \in \Irr}{V_{\xi} \otimes X_{\xi}})}$
 as the composition of the following natural isomorphisms
\begin{center}
$\begin{tikzcd}
\bigoplus_{\chi \in \Irr}{V_{\chi}}=V \ar[d, "\sum_{\chi \in \Irr}{\mathrm{id}_{V_{\chi}}\otimes \mathrm{id}_{X_{\chi}}}"'] \ar[r, "\eta_V"] & \bigoplus_{\chi \in \Irr}{\Hom_{\Gamma}(X_{\chi}, \bigoplus_{\xi \in \Irr}{V_{\xi}} \otimes X_{\xi})} \\
\bigoplus_{\chi,\xi \in \Irr}{{V_{\xi} \otimes \Hom_{\Gamma}(X_{\chi}, X_{\xi}})} \ar[ur, "\sim"']
\end{tikzcd}$
\end{center}

$\bullet$ $\eta_{V^f}: V^f \iso \bigoplus_{\chi \in \Irr}{\Hom_{\Gamma}(X_{\chi}, \bigoplus_{\xi \in \Irr}{V^f_{\xi} \otimes X_{\xi}})}$ defined in the same way as $\eta_V$\\

$\bullet$ $\eta_x: x \mapsto \eta_V \circ x \circ \eta_V^{-1}$\\

$\bullet$ $\eta_{v^1}: v^1 \mapsto \eta_{V} \circ v^1 \circ \eta_{V^f}^{-1}$\\

$\bullet$ $\eta_{v^2}: v^2 \mapsto \eta_{V^f} \circ v^2 \circ \eta_V^{-1}$.\\
\\
\noindent Let us show that $\mathcal{G}(\mathcal{F}(x)) \circ \eta_V = \eta_V \circ x$. If $\chi \in \Irr$ and $v \in V_{\chi}$, then ${\eta_V(v)=v\otimes \mathrm{id}_{X_{\chi}}}$. If we take $h \in \overline{E}$ such that $h'=\chi$, then we have
\begin{equation*}
\mathcal{G}(\mathcal{F}(x))_h(\eta_V(v))=[(y_h\otimes x_h) \otimes (v \otimes \mathrm{id}_{X_{\chi}} \otimes \mathrm{id}_{\Xstd}) \circ \tilde{y}_h] \in \Hom_{\Gamma}(X_{h''}, \bigoplus_{\xi \in \Irr}{V_{\xi} \otimes X_{\xi}}).
\end{equation*}
Using that $\tilde{y}_h$ is a section of $y_h$, we have ${(y_h\otimes x_h) \otimes (v \otimes \mathrm{id}_{X_{\chi}} \otimes \mathrm{id}_{\Xstd}) \circ \tilde{y}_h=\eta_V(x(v))}$.\\
Furthermore, we have by construction that
\begin{equation*}
\eta_{V} \circ v^1 = (\mathcal{G}(\mathcal{F}(v^1))) \circ \eta_{V^f}
\end{equation*}
and that
\begin{equation*}
\eta_{V^f} \circ v^2 = (\mathcal{G}(\mathcal{F}(v^2))) \circ \eta_{V}.
\end{equation*}
\end{proof}

\begin{proof}[Proof of Proposition \ref{link_momentum}]
Take $(x,v^1,v^2) \in \RepQ_{V,V^f}$. We want to show that
\begin{equation*}
\sum_{\chi \in \Irr}{\mu_\varepsilon(x,v^1,v^2)_{\chi} \otimes \frac{1}{\delta_{\chi}}\mathrm{id}_{X_{\chi}}} = \mu\left(F_{y}(x,v^1,v^2)\right).
\end{equation*}
Expanding the right-hand side, we get
\begin{align*}
&\sum_{h \in \overline{E}}{\lphoe_{h} \lphoe_{\bar{h}} x_{\bar{h}} \otimes [y_{\bar{h}}]_1 \circ x_{h} \otimes [y_{h}]_2 - x_{\bar{h}} \otimes [y_{\bar{h}}]_2 \circ x_{h} \otimes [y_{h}]_1} + \sum_{\chi \in \Irr}{v^1_{\chi}v^2_{\chi}\otimes \frac{1}{\delta_{\chi}}\mathrm{id}_{X_{\chi}}}\\
&=\sum_{h \in \overline{E}}{\lphoe_{h} \lphoe_{\bar{h}} x_{\bar{h}}x_h \otimes \left([y_{\bar{h}}]_1[y_{h}]_2 - [y_{\bar{h}}]_2 [y_{h}]_1 \right)} + \sum_{\chi \in \Irr}{v^1_{\chi}v^2_{\chi}\otimes \frac{1}{\delta_{\chi}}\mathrm{id}_{X_{\chi}}}
\end{align*}
Using the fact that $\omega$ is nondegenerate, there exists $\lphoe_h^0 \in \C^*$ for each $h \in \overline{E}$, such that
\begin{equation*}
[y_{\bar{h}}^0]_1[y_{h}^0]_2 - [y_{\bar{h}}^0]_2 [y_{h}^0]_1=\lphoe_h^0 \mathrm{id}_{X_{h'}}.
\end{equation*}
By construction of the constants $\lphoe_h$, we have the following relation for all $h \in \overline{E}$
\begin{equation*}
\varepsilon(h)= \lphoe_h\lphoe_{\bar{h}}\lphoe_{\bar{h}}^0\delta_{h''}.
\end{equation*}
Summing it all up, the right-hand side gives
\begin{equation*}
\sum_{h \in \overline{E}}{\varepsilon(h) x_{h}x_{\bar{h}} \otimes \frac{1}{\delta_{h''}}\mathrm{id}_{X_{h''}}} + \sum_{\chi \in \Irr}{v^1_{\chi}v^2_{\chi}\otimes \frac{1}{\delta_{\chi}}\mathrm{id}_{X_{\chi}}}.
\end{equation*}
Which is exactly what we wanted to show.
For the other square, recall that
\begin{center}
$\epsilon_M\circ \mathcal{F}\left(\mathcal{G}(\Delta)\right)\circ \left(\mathrm{id}_{\Xstd}\otimes \epsilon_M^{-1}\right)=\Delta$,\\
$\epsilon_M\circ  \mathcal{F}\left(\mathcal{G}(Z_1)\right) \circ \epsilon_{M^f}^{-1}=Z_1$,\\
$\epsilon_{M^f}\circ \mathcal{F}\left(\mathcal{G}(Z_2)\right) \circ \epsilon_M^{-1}= Z_2$.
\end{center}
Using what we just proved we have
\begin{equation*}
\sum_{\chi \in \Irr}{\mu_{\varepsilon}\left(\mathcal{G}(\Delta,Z_1,Z_2)_{\chi}\right)\otimes \frac{1}{\delta_{\chi}}\mathrm{id}_{X_{\chi}}}= \mu\left(\mathcal{F}\left(\mathcal{G}(\Delta,Z_1,Z_2)\right)\right)
\end{equation*}
Since $\mu\left(\mathcal{F}\left(\mathcal{G}(\Delta,Z_1,Z_2)\right)\right)=\epsilon_M^{-1}\mu(\Delta,Z_1,Z_2)\epsilon_M$, we have shown that the second square also commutes.
\end{proof}

\printbibliography

\Addresses

\end{document}